\documentclass{amsart}
\usepackage{amsmath,amsthm,amssymb,amsfonts,amscd}
\usepackage{comment}

\theoremstyle{plain}
\newtheorem{Thm}{Theorem}[section]
\newtheorem{Prop}[Thm]{Proposition}

\newtheorem{Coro}[Thm]{Corollary}
\newtheorem{Lem}[Thm]{Lemma}

\theoremstyle{definition}
\newtheorem{Defn}[Thm]{Definition}
\newtheorem{Conj}[Thm]{Conjecture}
\newtheorem{Ex}[Thm]{Example}

\newtheorem{Remk}[Thm]{Remark}

\newcommand{\ZZ}{\mathbb{Z}}

\newcommand{\PP}{\mathbb{P}}

\newcommand{\II}{\mathcal{I}}

\newcommand{\G}{\mathcal{G}}
\newcommand{\SSS}{\mathcal{S}}
\newcommand{\MO}{M_\omega}
\newcommand{\HB}{h_\bullet}
\newcommand{\GB}{g_\bullet}
\newcommand{\HIB}{h^{(i)}_\bullet}
\newcommand{\gin}{\mathrm{gin}}
\newcommand{\OT}{\tilde{\omega}}
\newcommand{\AT}{\tilde{\alpha}}

\newcommand{\RA}{\rightarrow}

\numberwithin{equation}{section}


\begin{document}
\allowdisplaybreaks
\title[Almost reverse lexicographic ideals and Fr\"{o}berg sequences]
      {Almost reverse lexicographic ideals and Fr\"{o}berg sequences}

\author{Jung Pil Park}
\address{Algebraic Structures and Its Application Research Center(ASARC),KAIST, Daejeon 305-340, South Korea }
\email{jppark@math.snu.ac.kr} %
\thanks{This research was supported by Basic Science Research Program through the National Research Foundation of Korea(NRF) funded by the Ministry of Education, Science and Technology (2009-0063180)}

\date{\today}

\begin{abstract}
   We study almost reverse lexicographic ideals in a polynomial ring over a field of arbitrary characteristic. We give a criterion for a given sequence of nonnegative integers to be the Hilbert function of an almost reverse lexicographic ideal in the polynomial ring. Then it will be shown that every Fr\"{o}berg sequence satisfies this criterion.
\end{abstract}

\maketitle

%%%%%%%%%%%%%%%%%%%%%%%%%%%%%%%%%%%%%%%%%%%%%%%%%%%%%%%%%%%%%%%%%%%%%%%%%%%%%%%%%%%%%%%%%%%%%%%%%%%%%%%%%%%%%%%%%
%%%%%%%%%%%%%%%%%%%%%%%%%%%%%%%%%%%%%%%%%%%%%%%%%%%%%%%%%%%%%%%%%%%%%%%%%%%%%%%%%%%%%%%%%%%%%%%%%%%%%%%%%%%%%%%%%
\section{\sc Introduction}
%%%%%%%%%%%%%%%%%%%%%%%%%%%%%%%%%%%%%%%%%%%%%%%%%%%%%%%%%%%%%%%%%%%%%%%%%%%%%%%%%%%%%%%%%%%%%%%%%%%%%%%%%%%%%%%%%
%%%%%%%%%%%%%%%%%%%%%%%%%%%%%%%%%%%%%%%%%%%%%%%%%%%%%%%%%%%%%%%%%%%%%%%%%%%%%%%%%%%%%%%%%%%%%%%%%%%%%%%%%%%%%%%%%
  In 1927, Macaulay \cite{Ma} showed that for a given sequence $\HB=(h_0,h_1,\ldots)$ of nonnegative integers $h_i$ with $h_0=1$, there is a homogeneous ideal $I$ in a polynomial ring $R=k[x_1,\ldots,x_n]$ over a field $k$ such that its Hilbert function $H(R/I,i) := \dim_k (R/I)_i$ is equal to $h_i$ for each $i \ge 0$ if and only if the sequence satisfies the Macaulay bound. To prove this beautiful theorem, Macaulay used special monomial ideals, lex-segment ideals. Since then, the study of the shape of the Hilbert function reflecting the properties of a given standard $k$-algebra $R/I$ became a very active research area.

  Fr\"{o}berg conjecture can be considered as one of the most outstanding problems in this research area. Fr\"{o}berg \cite{Fr} conjectured in 1985 that
  \begin{Conj}[Fr\"{o}berg]
    If $I$ is an ideal generated by generic forms $F_1,\ldots,F_r$ in $R$ with $\deg F_i = d_i$, then the Hilbert series $S_{R/I}(z) := \sum_{i=0}^{\infty} H(R/I,i) z^i \in \ZZ[[z]]$ of $R/I$ is given by
    \[
        S_{R/I}(z) = \left|
          \frac{\prod_{i=1}^{r}(1-z^{d_i})}{(1-z)^n}\right|,
    \]
    where for a formal power series $\sum a_iz^i$ with integer coefficients we
    let $|\sum a_iz^i| = \sum b_iz^i$ with $b_i = a_i$ if $a_j > 0$ for
    all $0 \le j \le i$, and $b_i = 0$ if $a_j = 0$ for some $j \le i$.
  \end{Conj}

  This conjecture is true if $r \le n$ because in this case every generic forms consist a regular sequence. Stanley \cite{St} and Watanabe \cite{Wa} showed that the conjecture holds if $r=n+1$ when the characteristic of the base field $k$ is $0$. And Fr\"{o}berg \cite{Fr} proved the conjecture holds for the case $n=2$, and Anick \cite{An} showed for the case $n=3$ in an infinite field, respectively. Hochster and Laksov \cite{HL} showed that $H(R/I,d+1)$ is given by the formula, where $I$ is the ideal generated by $m$ generic forms $F_1,\ldots,F_m$ such that $\deg F_{m} = 1 + \max\{\deg F_i | 1 \le i \le m-1 \}$. Aubry \cite{Au} extended their work to certain other values.

  To study Fr\"{o}berg conjecture, many approaching methods have been developed. The strong Lefschetz property can be said to be one of the most useful methods. For an Artinian ring $A=R/I$, we say that $A$ has the strong (resp.\ weak) Lefschetz property if there exists a linear form $L$ in $R$ such that the multiplication $\times L^{i} : (R/I)_d \rightarrow (R/I)_{d+i}$ is either injective or surjective for each $d$ and $i$(resp.\ for $i=1$). Suppose that an Artinian ring $R/I$ has the strong Lefschetz property. For a degree $d$ form $F \in R$ we have the following exact sequence
  \[
      0 \RA ((I:F)/I)_{t-d} \RA (R/I)_{t-d} \xrightarrow[\quad]{\times F} (R/I)_{t}
         \rightarrow (R/I+(F))_{t}
         \rightarrow 0.
  \]
  If $F$ is generic, then the Hilbert function of $R/I+(F)$ is given by
  \begin{equation*}
     H(R/I+(F),t) = \max\{0, H(R/I,t)-H(R/I,t-d)\}.
  \end{equation*}
  Hence the Hilbert series of $R/I+(F)$ satisfies
  \[
     S_{R/I+(F)}(z) = |(1-z^{d})S_{R/I}(z)|.
  \]
  This shows that the study for the strong Lefschetz property of a standard graded $k$-algebra defined by generic forms is closely related with the study for Fr\"{o}berg conjecture. There are various results on Lefschetz properties \cite{Ci,HMNW,HP,HW,MM}.

  In 2004, Wiebe \cite{Wi} showed that an Artinian ideal $R/I$ has the weak (resp.\ strong) Lefschetz property if and only if for the generic initial ideal $\gin(I)$ of $I$ with respect to the degree reverse lexicographic order, $R/\gin(I)$ has the weak (resp.\ strong) Lefschetz property. (For generic initial ideals, you can refer some literatures, e.g.\ given by Bayer and Stillman \cite{BS}, Conca \cite{C}, Eisenbud \cite{Ei}, Fl{\o}ystad \cite{Fl}, Galligo \cite{Ga}, Green \cite{Gr}, Cho et al.\ \cite{CCP}, or Cho and the author \cite{CP} for basic definition and theorems.) Using the result of Wiebe, Ahn et al.\ \cite{ACP} found an equivalent condition for $R/I$ to have the weak or strong Lefschetz properties in terms of the minimal system of generators of $\gin(I)$ in 2007. Cho and the author \cite{CP} showed that if $\gin(I)$ is almost reverse lexicographic(See Definition \ref{Def:ARL_Ideal}), then $R/I$ satisfies the strong Lefschetz property. They also showed that for an ideal $I$ generated by generic forms $F_1,\ldots,F_m$ if there is an ideal $K$ generated by homogeneous polynomials $G_1,\ldots,G_m$ with $\deg F_i = \deg G_i$ such that (1) $H(R/I,d) = H(R/K,d)$ for all $d$ and (2) $\gin(K)$ is almost reverse lexicographic, then $\gin(I)$ is also almost reverse lexicographic.

  This shows that the study about almost reverse lexicographic ideal is needed to solve Fr\"{o}berg conjecture.
  In Section 2, we study almost reverse lexicographic ideal $I$ in a polynomial ring $R$ over a field of arbitrary characteristic intensively. In particular, we give the minimal system of generators of $I$ definitely.

  In Section 3, we study Hilbert functions of almost reverse lexicographic ideals. Then we give a criterion for a given sequence $\HB$ of nonnegative integers to be induced from an almost reverse lexicographic ideal in a polynomial ring. In particular, we give an algorithm to obtain the almost reverse lexicographic ideal from a given sequence satisfying the condition.

  In the last section, we show that every Fr\"{o}berg sequence (See Definition \ref{Def:FRO_Seq}) is induced from an almost reverse lexicographic ideal in a polynomial ring.

  \vskip 2cm
%%%%%%%%%%%%%%%%%%%%%%%%%%%%%%%%%%%%%%%%%%%%%%%%%%%%%%%%%%%%%%%%%%%%%%%%%%%%%%%%%%%%%%%%%%%%%%%%%%%%%%%%%%%%%%
%%%%%%%%%%%%%%                  End of Introduction
%%%%%%%%%%%%%%%%%%%%%%%%%%%%%%%%%%%%%%%%%%%%%%%%%%%%%%%%%%%%%%%%%%%%%%%%%%%%%%%%%%%%%%%%%%%%%%%%%%%%%%%%%%%%%%

%%%%%%%%%%%%%%%%%%%%%%%%%%%%%%%%%%%%%%%%%%%%%%%%%%%%%%%%%%%%%%%%%%%%%%%%%%%%%%%%%%%%%%%%%%%%%%%%%%%%%%%%%%%%%%%%%
%%%%%%%%%%%%%%%%%%%%%%%%%%%%%%%%%%%%%%%%%%%%%%%%%%%%%%%%%%%%%%%%%%%%%%%%%%%%%%%%%%%%%%%%%%%%%%%%%%%%%%%%%%%%%%%%%
\section{\sc Minimal System of Generators}
%%%%%%%%%%%%%%%%%%%%%%%%%%%%%%%%%%%%%%%%%%%%%%%%%%%%%%%%%%%%%%%%%%%%%%%%%%%%%%%%%%%%%%%%%%%%%%%%%%%%%%%%%%%%%%%%%
%%%%%%%%%%%%%%%%%%%%%%%%%%%%%%%%%%%%%%%%%%%%%%%%%%%%%%%%%%%%%%%%%%%%%%%%%%%%%%%%%%%%%%%%%%%%%%%%%%%%%%%%%%%%%%%%%

   Let $R=k[x_1, \ldots, x_n]$ be a polynomial ring over a field $k$ of arbitrary characteristic. As a term order on the set of the monomials in $R$, we use only the degree reverse lexicographic order, which is defined as follows: for monomials $M=x_1^{\alpha_1}\cdots x_n^{\alpha_n}$ and $N = x_1^{\beta_1} \cdots x_n^{\beta_n}$ in $R$, we say that $M > N$ if and only if either (1) $\deg M > \deg N$ or (2) $\deg M = \deg N$ and there is an integer $1 \le s \le n$ such that $\alpha_i = \beta_i$ for all $i \ge s$ but $\alpha_s < \beta_s$. (We note that if $M > N$, then we always exclude the possibility $M = N$.)

   A monomial ideal $I$ in $R$ is said to be strongly stable if for any monomial $M$ in $R$, $x_iM \in I$ implies $x_j M \in I$ whenever $1\le j < i$. Let $I$ be a strongly stable monomial ideal in $R$. We denote by $\G(I)$, or simply by $\G$, the minimal system of generators of $I$. For a monomial $M = x_1^{\alpha_1} \cdots x_n^{\alpha_n}$ in $R$, we set $\max M =\max \{ i | \alpha_i > 0 \}$. We will use multi-index notation, hence for any integer $s$ with $1 \le s\le n$, if $\alpha=(\alpha_1, \ldots, \alpha_s)$ is an element in $\ZZ^s_{\ge 0}$, then $x^\alpha$ means the monomial $x_1^{\alpha_1} \cdots x_{s}^{\alpha_s}$. Next we set $|\alpha| = \sum_{j=1}^{s} \alpha_j = \deg x^{\alpha}$. Then we extend the term order $<$ defined on the set of the monomials in $R$ to each set $\ZZ^s_{\ge 0}$ as follows: For $\alpha=(\alpha_1, \ldots, \alpha_s), \beta=(\beta_1, \ldots, \beta_s) \in \mathbb{Z}^s_{\ge 0}$, we define
   \[
       \alpha \le \beta \text{ if and only if } x^\alpha \le x^\beta.
   \]
   In particular, $\alpha \le \beta$ implies $|\alpha| \le |\beta|$.

   \begin{Defn}
      Suppose that $I$ is a strongly stable ideal in a polynomial ring $R=k[x_1,\ldots,x_n]$. By the last generator of $I$, denoted by $\MO(I)$ or simply by $\MO$, we mean the monomial in $\G$ satisfying the following two conditions:
      \begin{enumerate}
      \item $\deg \MO \ge \deg M$ for any $M \in \G$, and
      \item $\MO \le N$ for any $N \in \G$ with $\deg N = \deg \MO$.
      \end{enumerate}
   \end{Defn}

   Let $I$ be a strongly stable ideal in $R$. Suppose that $\mu = \max \MO$. To find out the minimal system $\G$ of generators of $I$, we define a positive integer $f_1$ and functions $f_i:\mathbb{Z}^{i-1}_{\ge 0} \longrightarrow \mathbb{Z}_{\ge 0} \cup \{ \infty \}$ for each $2 \le i \le \mu$ as follows:
   \begin{align}\label{Eqn:f's}
      \begin{split}
        f_1=&\min \{t \ | \, x^t_1 \in I \} \text{, and }\\
        f_i(\alpha_1, \ldots, \alpha_{i-1}) =& \min \{ t \ | \,
                  x^{\alpha_1}_1 \cdots x^{\alpha_{i-1}}_{i-1} x^{t}_i \in I \} \text{ for } (\alpha_1, \ldots, \alpha_{i-1}) \in \ZZ^{i-1}_{\ge 0}.
      \end{split}
   \end{align}

   \begin{Remk}\label{Remk:When_F_be_Zero}
       Let $s$ be an integer  with $1 \le s \le \mu-1$. For $\alpha=(\alpha_1, \cdots, \alpha_s) \in \ZZ^{s}_{\ge 0}$, note that $f_{s+1}(\alpha) = 0$ if and only if $x^{\alpha} \in I$.
   \end{Remk}

   To avoid endlessly repeating the hypotheses, we will assume throughout this article that the $f_i$'s are defined for a given strongly stable ideal $I$ as like {\rm(\ref{Eqn:f's})} where $\mu = \max \MO$. If there is a number of strongly stable ideals to be considered, then we use the notation $f^{I}_i$ for distinction. We also assume that $\mu \ge 2$ unless otherwise specified. Because if $\mu=1$, then $\G(I)=\{x_1^{f_1}\}$, and every theorems in this section are trivial.

   The following lemmas, introduced by Cho et al.\ \cite{CCP}, help us to determine the minimal system of generators of an almost reverse lexicographic ideal. They showed that if $I$ is an ideal in $R$ such that $R/I$ is a Cohen-Macaulay ring with $\dim R/I = r$, then the generic initial ideal $\gin(I)$ of $I$ is completely determined by the values $f_1$,$f_2,\ldots,f_{n-r}$. In particular, when they proved the following lemma, they used just two facts that (1) $x_{n-r}^t$ is contained in $\gin(I)$ for some $t > 0$ since $\dim R/I = r$, and that (2) $\gin(I)$ is strongly stable. Thus the proof given in that paper works in our case. (See the papers \cite{CCP,HS} for more information on the relations between dimension, depth, and the minimal systems of generators of strongly stable monomial ideals.)

   \begin{Lem}\cite{CCP}\label{Lem:Property_of_F's_Orig}
       Let $I$ be a strongly stable ideal in $R=k[x_1,\ldots,x_n]$. Suppose that the last generator of $I$ is $M_{\omega} = x^{\omega}$ for some $\omega=(\omega_1,\ldots,\omega_{\mu}) \in \ZZ^{\mu}_{\ge 0}$ with $\omega_{\mu} > 0$, and that $I$ contains the monomial $x_{\mu-1}^t$ for some $t > 0$. Let $1 \le i \le \mu-2$. For any $\alpha=(\alpha_1,\ldots,\alpha_{i}) \in \ZZ^{i}$ such that $0 \le \alpha_1 < f_1$ and $0 \le \alpha_{j} < f_{j}(\alpha_1, \ldots, \alpha_{j-1})$ whenever $2 \le j\le i$, we have
       \begin{enumerate}
          \item $0 < f_{i+1}(\alpha) < \infty$,
          \item $x^\alpha x^{f_{i+1}(\alpha)}_{i+1} = x_1^{\alpha_1} \cdots x_{i}^{\alpha_{i}} x^{f_{i+1}(\alpha)}_{i+1} \in \G$, and
          \item If $\alpha_j \ge 1$ for some $1 \le j \le i-1$, then
             \[
                 f_{i+1}(\alpha_1, \ldots, \alpha_j, \ldots, \alpha_{i}) + 1  \le
                 f_{i+1}(\alpha_1, \ldots, \alpha_j-1, \ldots, \alpha_{i}).
             \]
       \end{enumerate}
   \end{Lem}
   \begin{proof}
     See Lemma 3.2 in the paper \cite{CCP}.
   \end{proof}

   Although the following lemma is also given in the paper \cite{CCP}, we give its proof here, because they used the hypothesis $R/I$ is Cohen-Macaulay to assure that there is no minimal generator $M$ of $I$ with $\max M > n-r$ under the same assumption as above.

   \begin{Lem}\cite{CCP}\label{Lem:MinGen_Basic}
     Let $I$ be a strongly stable ideal in $R=k[x_1,\ldots,x_n]$. Suppose that the last generator of $I$ is $M_{\omega} = x^{\omega}$ for some $\omega=(\omega_1,\ldots,\omega_{\mu}) \in \ZZ^{\mu}_{\ge 0}$ with $\omega_{\mu} > 0$. If $I$ contains the monomial $x_{\mu-1}^t$ for some $t > 0$, then
     \begin{align*}
         \{M \in \G|\max M \le & \mu-1\} \\
             &= \{x_1^{f_1}\} \cup
               \bigcup_{i=1}^{\mu-2}
                \left\{x^{\alpha}x_{i+1}^{f_{i+1}(\alpha)}
                \left| \begin{array}{l}
                         \alpha=(\alpha_1,\ldots,\alpha_{i}) \in \ZZ^{i}_{\ge 0},\\
                         0 \le \alpha_1 < f_1, \text{ and }\\
                         0 \le \alpha_{j} < f_{j}(\alpha_1, \ldots, \alpha_{j-1})\\
                         \text{for each} \ 2 \le j \le i
                       \end{array}
                \right.
                \right\}.
     \end{align*}
   \end{Lem}
   \begin{proof}
     By Lemma \ref{Lem:Property_of_F's_Orig}, it is enough to show that the set in the left hand side is a subset of the one in the right. Let $M$ be a minimal generator of $I$. If $\max M = 1$, then $M=x_1^{f_1}$ by the definition of $f_1$. So we may assume that $\max M \ge 2$. Suppose that $M=x^{\alpha}x_{i+1}^{\beta}$ for some $\alpha=(\alpha_1,\ldots,\alpha_{i}) \in \ZZ^{i}_{\ge 0}$ with $1 \le i \le \mu-2$ and $\beta > 0$. We claim that $\alpha_j < f_{j}(\alpha_1,\ldots,\alpha_{j-1})$ for any $1 \le j \le i$. If our claim is true, then $\beta \ge f_{i+1}(\alpha)$ by the definition of $f_{i+1}$. Since $M$ is a minimal generator, it also follows that $\beta \le f_{i+1}(\alpha)$. Thus $M$ is contained in the right hand side set.

     To prove the claim, suppose that there is an integer $j$ with $1 \le j \le i$ such that $f_j(\alpha_1,\ldots,\alpha_{j-1}) \ge \alpha_{j}$, then $x_1^{\alpha_1}\cdots x_{j}^{\alpha_j}$ belongs to $I$ by the definition of $f_j$. This contradicts that $M$ is a minimal generator of $I$. So we are done.
   \end{proof}

   Let $I$ be a strongly stable ideal in $R=k[x_1,\ldots,x_n]$. Suppose that the last generator of $I$ is $M_{\omega} = x^{\omega}$ for some $\omega=(\omega_1,\ldots,\omega_{\mu}) \in \ZZ^{\mu}_{\ge 0}$ with $\omega_{\mu} > 0$, and that $I$ contains the monomial $x_{\mu-1}^t$ for some $t > 0$. For each $1 \le i \le \mu-2$, we define
   \begin{align}\label{Set:II's}
      \begin{split}
       \II_i &= \left\{(\alpha_1,\ldots,\alpha_{i}) \in \ZZ^{i}_{\ge 0}
                \left| \begin{array}{l}
                         0 \le \alpha_1 < f_1, \ \text{ and } \\
                         0 \le \alpha_{j} < f_{j}(\alpha_1, \ldots, \alpha_{j-1}) \\
                         \text{for each} \ 2 \le j \le i
                       \end{array}
                \right.
               \right\}, \text{ and let} \\
       \II_{\mu-1} &= \left\{(\alpha_1,\ldots,\alpha_{\mu-1}) \in \ZZ^{\mu-1}_{\ge 0}
                 \left| \begin{array}{l}
                         0 \le \alpha_1 < f_1, \\
                         0 \le \alpha_{j} < f_{j}(\alpha_1, \ldots, \alpha_{j-1}) \\
                         \text{for each} \ 2 \le j \le \mu-1,  \ \text{and} \\
                         (\alpha_1,\ldots,\alpha_{\mu-1}) \ge (\omega_1,\ldots,\omega_{\mu-1})
                       \end{array}
                \right.  \right\}.
      \end{split}
   \end{align}

   As like the $f_i$'s, we maintain the assumption that the $\II_i$'s are defined for each strongly stable ideal throughout this paper. If we want to distinguish the $\II_i$'s for two ideals $I$ and $J$, then we use the notations $\II^{I}_i$ and $\II^{J}_i$, respectively.

   The following lemma gives an upper bound for number of elements of the set $\{ M \in \G(I) \, | \, \max M = \max \MO\}$. It will be used in the main theorem of next section(See Theorem \ref{Thm:UAET_Imp_ARL_IDEAL}). For a set $S$, we denote by $|S|$ the number of elements in the set $S$.

   \begin{Lem}\label{Lem:NumberOflastGens}
      Let $I$ be a strongly stable ideal in $R=k[x_1,\ldots,x_n]$. Suppose that the last generator of $I$ is $M_{\omega} = x^{\omega}$ for some $\omega=(\omega_1,\ldots,\omega_{\mu}) \in \ZZ^{\mu}_{\ge 0}$ with $\omega_{\mu} > 0$, and that $I$ contains the monomial $x_{\mu-1}^t$ for some $t > 0$. Then we have
      \[
          |\II_{\mu-1}| \le \sum_{\alpha \in \II_{\mu-2}} f_{\mu-1}(\alpha).
      \]
      The equality holds if and only if $\MO = x_{\mu}^{\omega_{\mu}}$, i.e.\ $\omega=(0,\ldots,0,\omega_{\mu})$. In particular, the set $\II_{\mu-1}$ is a finite set.
   \end{Lem}
   \begin{proof}
   We let
   \[
       T =   \left\{(\alpha_1,\ldots,\alpha_{\mu-2},t) \in \ZZ^{\mu-1}_{\ge 0}
                  \left| \begin{array}{l}
                         \alpha = (\alpha_1,\ldots,\alpha_{\mu-2}) \in \II_{\mu-2}, \\
                         0 \le t < f_{\mu-1}(\alpha)
                       \end{array}
                \right.  \right\}.
   \]
   Then it suffices to show $\II_{\mu-1} \subset T$. But that is clear from the definition of $\II_{\mu-1}$.
   Indeed, if $\alpha=(\alpha_1,\ldots,\alpha_{\mu-1})$ is an element in $\II_{\mu-1}$, then the $(\mu-2)$-tuple $\AT := (\alpha_1,\ldots,\alpha_{\mu-2})$ belongs to $\II_{\mu-2}$ because $\alpha_i < f_{i}(\alpha_1,\ldots,\alpha_{i-1})$ for each $i$. Furthermore since $\alpha_{\mu-1} < f_{\mu-1}(\AT)$, we have $\alpha \in T$.

   Note that the equality holds if and only if $\II_{\mu-1} = T$. This occurs if and only if any element $(\alpha_1,\ldots,\alpha_{\mu-1})$ in $T$ is greater than or equal to $\OT=(\omega_1,\ldots,\omega_{\mu-1})$ by the definition of $\II_{\mu-1}$. Since $f_{\mu-1}(0) > 0$, $(0,\ldots,0) \in \ZZ^{\mu-1}$ is also contained in $T$. This implies that $\OT$ must be $(0,\ldots,0)$. So we are done.

   We note that $|\G(I)| < \infty$. This implies that $|\II_{i}| < \infty$ for all $1 \le i \le \mu-2$, because there is a one to one correspondence between $\II_i$ and the set $\{M \in \G(I) | \max M = i+1\}$ via the map $\alpha \mapsto x^{\alpha}x_{i+1}^{f_{i+1}(\alpha)}$ for each $1 \le i < \mu-1$. Hence the last statement follows from Lemma \ref{Lem:Property_of_F's_Orig} (1).
   \end{proof}

   \begin{Remk}
     Under the same assumption in Lemma \ref{Lem:NumberOflastGens}, the set $\II_{\mu-1}$ is a finite set. Furthermore $f_{\mu}(\alpha) > 0$ for any $\alpha \in \II_{\mu-1}$ by Remark \ref{Remk:When_F_be_Zero}. But as shown in the following example, we can have $f_{\mu}(\alpha)=\infty$.
   \end{Remk}

   \begin{Ex}
     \begin{enumerate}
       \item Let $I$ be the ideal in $R=k[x,y,z]$ generated by the set
           \[
              \G = \left\{
                          \begin{array}{cccc}
                              x^3,& x^2y,  & xy^3,   & y^5, \\
                                  & x^2z^5,& xy^2z^3,& y^4z^2
                              \end{array}
                   \right\}.
           \]
           Then one can see that $I$ is almost reverse lexicographic. Note that $\MO = x^2z^5$, $\II_1 = \{2,1,0\}$, and $\II_2 = \{(0,4),(1,2),(0,3),(2,0)\}$. But we have $f_3(0,3) = \min \{d|y^3z^d \in I \} = \infty$.
       \item If $I$ doesn't contain $x_{\mu-1}^t$ for some $t>0$, then the set $\II_{\mu-1}$ is not a finite set. Consider the ideal $I$ generated by $\{x^2, xy^2, xyz^2 \}$.  Then the ideal $I$ is strongly stable, and $\MO = xyz^2$. Note that $\II_1=\{1,0\}$ and $\OT = (1,1) \in \II_2$. By the definition of $\II_2$, any $2$-tuple $(0,b)$ is contained in $\II_2$ if $b < f_2(0)$ and $(0,b) > \OT$. Since $f_2(0) = \infty$, $(0,b) \in \II_2$ for any $b \ge 3$. Hence $\II_2$ is an infinite set.
     \end{enumerate}
   \end{Ex}

   \begin{Defn}\label{Def:ARL_Ideal}
      A monomial ideal $I$ in a polynomial ring $R=k[x_1,\ldots,x_n]$ is said to be almost reverse lexicographic if for any monomial $M$ in $R$ and any minimal generator $N$ of $I$ with $\deg N = \deg M$, $I$ contains $M$ whenever $M > N$.
   \end{Defn}

   Suppose that $I$ is an almost reverse lexicographic ideal in $R$. In this section we give the minimal system $\G$ of generators of $I$ in a close form. As a first step, we claim that $I$ is strongly stable. Indeed, for a monomial $M$ in $I$, suppose that $x_i$ divides $M$ and $j < i$. Then there is $N \in \G$ such that $N$ divides $M$. If $x_i$ cannot divide $N$, then $N$ divides $M/x_i$, so $x_j(M/x_i)$ is contained in $I$. On the other hand, if $x_i$ divides $N$, then $x_j(N/x_i) > N$. Since $I$ is almost reverse lexicographic, $I$ contains $x_j(N/x_i)$, and hence also contains $x_j(M/x_i)$ .

   Let $\MO= x^{\omega}$ be the last generator of $I$ for some $\omega=(\omega_1,\ldots,\omega_{\mu})$ with $\omega_{\mu}>0$. First we note that $x_{\mu-1}^{|\omega|} > \MO$. Since $I$ is almost reverse lexicographic, $x_{\mu-1}^{|\omega|}$ is contained in $I$. Hence we have the finite set $\II_i$ for each $1 \le i \le \mu-1$. Now we can have the following lemma similar to Lemma \ref{Lem:Property_of_F's_Orig}, which gives the minimal system of generators of $I$.

   \begin{Lem} \label{Lem:Property_Of_F's_Ext}
       Let $I$ be an almost reverse lexicographic ideal in a polynomial ring $R=k[x_1,\ldots,x_n]$. Suppose that the last generator of $I$ is $\MO= x^{\omega}$ for some $\omega=(\omega_1,\ldots,\omega_{\mu})$ with $\omega_{\mu}>0$. Let $1\le i \le \mu-1$.
       For any $\alpha=(\alpha_1,\ldots,\alpha_{i}) \in \II_{i}$ we have
       \begin{enumerate}
         \item $0<f_{i+1}(\alpha) <\infty$,
         \item $x^\alpha x^{f_{i+1}(\alpha)}_{i+1} = x_1^{\alpha_1} \cdots x_{i}^{\alpha_{i}} x^{f_{i+1}(\alpha)}_{i+1} \in \G$, and
         \item If $\alpha_j \ge 1$ for some $1 \le j \le i$, then
           \[
               f_{i+1}(\alpha_1, \ldots, \alpha_j, \ldots, \alpha_{i}) + 1\le
               f_{i+1}(\alpha_1, \ldots, \alpha_j-1, \ldots, \alpha_{i}).
           \]
       \end{enumerate}
   \end{Lem}
   \begin{proof}
     As shown in the discussion preceding this lemma, $I$ is strongly stable, and $x_{\mu-1}^{|\omega|} \in I$. Hence the assertions follow if $1 \le i \le \mu-2$ by Lemma \ref{Lem:Property_of_F's_Orig}. Thus it suffices to prove for only the case $i=\mu-1$. Then $\alpha=(\alpha_1,\ldots,\alpha_{\mu-1}) \in \II_{\mu-1}$, and hence $\alpha \ge (\omega_1,\ldots,\omega_{\mu-1})$.

     By the choice of the $\alpha_i$'s, we have $x_1^{\alpha_1} \cdots x_{\mu-1}^{\alpha_{\mu-1}} \notin I$. It follows that $f_{\mu}(\alpha) > 0$ by Remark \ref{Remk:When_F_be_Zero}. Recall that $x_{\mu-1}^{|\omega|} \in I$. Since $I$ is strongly stable, if $|\alpha| \ge |\omega| = \deg \MO$, then $x_1^{\alpha_1} \cdots x_{\mu-1}^{\alpha_{\mu-1}} \in I$, a contradiction. This shows that $|\alpha| < |\omega|$. Set $M=x_1^{\alpha_1} \cdots x_{\mu-1}^{\alpha_{\mu-1}} x_{\mu}^{|\omega| - |\alpha|}$, then $\deg M = |\omega| = \deg \MO$. We claim that $M \ge \MO$. If our claim is true, then $I$ contains the monomial $M$, since $I$ is almost reverse lexicographic.  Therefore $f_{\mu}(\alpha) \le |\omega| - |\alpha| < \infty$.

     Since $\alpha > (\omega_1,\ldots,\omega_{\mu-1})$, $|\alpha| \ge |(\omega_1,\ldots,\omega_{\mu-1})| = |\omega|-\omega_\mu$, or equivalently, $|\omega| - |\alpha| \le \omega_{\mu}$. If $|\omega| - |\alpha| < \omega_{\mu}$, then $M \ge \MO$ clearly. So we may assume that $|\omega| - |\alpha| = \omega_{\mu}$. But then $|(\alpha_1,\ldots,\alpha_{\mu-1})| = |(\omega_1,\ldots,\omega_{\mu-1})|$, and hence $(\alpha_1,\ldots,\alpha_{\mu-1}) > (\omega_1,\ldots,\omega_{\mu-1})$. It follows $M \ge \MO$, so we are done.

     (2) follows from the choice of the $\alpha_i$'s and the definition of $f_{i+1}$.

     (3) is clear from the definition of $f_{i+1}$ and the strongly stableness of $I$.
   \end{proof}

   \begin{Prop}\label{Prop:MinGenOfARL}
      Let $I$ be an almost reverse lexicographic ideal in a polynomial ring $R=k[x_1,\ldots,x_n]$. Then the minimal system $\G$ of generator of $I$ is
      \begin{equation} \label{Set:MGS}
          \G = \{x_1^{f_1}\} \cup \bigcup_{i=1}^{\mu-1}
               \{
                  x^{\alpha_1}_1\cdots x^{\alpha_{i}}_{i} x^{f_{i+1}(\alpha)}_{i+1} = x^{\alpha}x^{f_{i+1}(\alpha)}_{i+1} \ | \
                  \alpha = (\alpha_1, \ldots, \alpha_{i}) \in \II_i
               \},
   \end{equation}
   where the last generator of $I$ is $\MO = x_1^{\omega_1} \cdots x_{\mu}^{\omega_{\mu}}$ for some $(\omega_1,\ldots, \omega_{\mu}) \in \mathbb{Z}^{\mu}_{\ge 0}$ with $\omega_{\mu} > 0$.
   \end{Prop}
   \begin{proof}
      If $\mu = 1$, then there is no generator other than $x_1^{f_1}$. So there is nothing to prove. Suppose $\mu \ge 2$. By Lemma \ref{Lem:MinGen_Basic} and \ref{Lem:Property_Of_F's_Ext}, it is enough to show that for any minimal generator $M=x_1^{\alpha_1} \cdots x_{\mu-1}^{\alpha_{\mu-1}}x_{\mu}^{\beta}$ with $\beta > 0$, the set $\II_{\mu-1}$ contains $\alpha=(\alpha_1,\ldots,\alpha_{\mu-1})$. Set $\OT=(\omega_1,\ldots,\omega_{\mu-1})$. Then $\OT \in \II_{\mu-1}$ and $\omega_{\mu} = f_{\mu}(\OT)$.

      If either $\alpha_1 \ge f_1$, or there is an integer $i$ with $2 \le i \le \mu-1$ such that $\alpha_i \ge f_{i}(\alpha_1,\ldots,\alpha_{i-1})$, then  $x_1^{\alpha_1} \cdots x_{i}^{\alpha_i}$ belongs to $I$ by the definition of $f_i$. But then the monomial $M=x^{\alpha}x_{\mu}^{\beta}$ can't be a minimal generator of $I$.

      Hence, in order to show that $\alpha \in \II_{\mu-1}$, it suffices to show $\alpha \ge \OT$. Suppose on the contrary that $\alpha < \OT$. Then $|\alpha| \le |\OT|$. Since $\MO$ is the last generator of $I$, $\deg M = |\alpha| + \beta \le \deg \MO = |\OT| + \omega_{\mu}$. But if $\deg M = \deg \MO$, then $\omega_{\mu} \le \beta$ because $|\alpha| \le |\OT|$. This implies $\MO > M$, which contradicts to the definition of $\MO$. This shows that $|\alpha| + \beta < |\OT| + \omega_{\mu}$. Since $\alpha < \OT$, there are only two possibilities:
      \begin{enumerate}
        \item[Case I.] $|\alpha| = |\OT|$: Then $\beta < \omega_{\mu}$. Note that the monomial $x^{\OT} x_{\mu}^{\beta}$ has the same degree with $M$ and is greater than $M$. Since $M$ is a minimal generator of $I$, the monomial $x^{\OT} x_{\mu}^{\beta}$ is contained in $I$. But this contradicts that $\MO$ is a minimal generator.
        \item[Case II.] $|\alpha| < |\OT|$: Then we can always choose a positive integer $t$ such that $\beta - \omega_{\mu} < t \le \min \{\beta,|\OT| - |\alpha|\}$. Since $x_{\mu-1}^{|\alpha|+t}x_{\mu}^{\beta-t} > M$, we have $x_{\mu-1}^{|\alpha|+t}x_{\mu}^{\beta-t} \in I$. Hence $x_{\mu-1}^{|\OT|}x_{\mu}^{\beta-t}$ is contained in $I$. It follows that $x^{\OT} x_{\mu}^{\beta-t}$ is also contained in $I$ because $I$ is strongly stable. But this contradicts to $\MO \in \G$, since $\beta-t < \omega_{\mu}$
      \end{enumerate}
      In any case, $\alpha < \OT$ induces a contradiction. Hence $\alpha \ge \OT$, so $\alpha$ is contained in the set $\II_{\mu-1}$.
   \end{proof}

   Now we will give a criterion for a strongly stable ideal to be almost reverse lexicographic.

   \begin{Prop}\label{Prop:EquivCondARL}
   Let $I$ be a strongly stable ideal in $R=k[x_1,\ldots,x_n]$. Suppose that the last generator of $I$ is $\MO = x_1^{\omega_1} \cdots x_{\mu}^{\omega_{\mu}}$ for some $(\omega_1,\ldots, \omega_{\mu}) \in \mathbb{Z}^{\mu}_{\ge 0}$ with $\omega_{\mu} > 0$, and that $x_{\mu-1}^t \in I$ for some $t >0$.
   Then $I$ is almost reverse lexicographic if and only if for any $1 \le i \le \mu-1$ and any $\alpha,\beta \in \II_{i}$ with $\alpha > \beta$,
   \[
       |\alpha| + f_{i+1}(\alpha) \le |\beta| + f_{i+1}(\beta).
   \]
   \end{Prop}
   \begin{proof}
       $(\Rightarrow):$ Fix $i$. Since $\alpha > \beta$, we have $x^{\alpha} x_{i+1}^{f_{i+1}(\beta)-|\alpha|+|\beta|} > x^{\beta} x_{i+1}^{f_{i+1}(\beta)}$. Since $I$ is almost reverse lexicographic, $x^{\alpha} x_{i+1}^{f_{i+1}(\beta)-|\alpha|+|\beta|}$ is contained in $I$. This shows that the assertion is true.

       $(\Leftarrow):$ Let $M$ and $N$ be monomials of the same degree. Suppose that $N$ is a minimal generator of $I$, and that $M > N$. We have to show that $M$ belongs to $I$. If $\max N = 1$, then $N = x_{1}^{f_1}$. In this case, there is no monomial of degree $f_1$, which is strictly bigger than $N$. So we may assume that $\max N > 1$. Since $M > N$, we have $\max M \le \max N$. Suppose that $M = x_1^{\beta_1} \cdots x_{s}^{\beta_{s}}x_{s+1}^{b}$ and $N = x_1^{\alpha_1} \cdots x_{s}^{\alpha_{s}}x_{s+1}^{a}$ for some $1 \le s \le \mu-1$. If we set $\alpha=(\alpha_1,\ldots,\alpha_{s}), \beta=(\beta_1,\ldots,\beta_{s})$, then $\alpha \in \II_s$ and $a = f_{s+1}(\alpha)$ by Proposition \ref{Prop:MinGenOfARL}.

       If there is an integer $j$ such that $1 \le j \le s$ and $\beta_j \ge f_{j}(\beta_1,\ldots,\beta_{j-1})$, then $M$ is contained in $I$ by the definition of $f_j$, so nothing is left to prove. Hence we may assume that $\beta_{j} < f_{j}(\beta_1,\ldots,\beta_{j-1})$ for any $1 \le j \le s$.

       Since $M = x^{\beta} x_{s+1}^b$ and $N = x^{\alpha} x_{s+1}^a$ are monomials of the same degree with $M > N$, we have $b \le a$ and $\beta > \alpha$. It follows from $\alpha \in \II_{s}$ that $\beta$ is also contained in $\II_s$. By the hypothesis then we have
       \[
           |\beta| + f_{s+1}(\beta) \le |\alpha| + f_{s+1}(\alpha) = |\alpha| + a = |\beta| + b.
       \]
       This shows that $M = x^{\beta}x_{s+1}^b \in I$ because $f_{s+1}(\beta) \le b$.
   \end{proof}

   Note that from the condition in ``if part" of the above proposition, we have $|\alpha| + f_{\mu}(\alpha) \le |\OT| + f_{\mu}(\OT) = \deg \MO$ for any $\alpha \in \II_{\mu-1}$, where $\OT = (\omega_1,\ldots,\omega_{\mu-1})$. Hence the case with $f_{\mu}(\alpha) = \infty$ won't happen anymore. The following example shows that the condition is not so trivial even if $0< f_{\mu}(\alpha) < \infty$ for any $\alpha \in \II_{\mu-1}$.

   \begin{Ex}
      Consider the strongly stable ideal $I$ in $R=k[x,y,z]$ generated by the set
          \[
             \G(I) = \left\{
                          \begin{array}{cccc}
                              x^3,& x^2y,  & xy^3,   & y^5, \\
                                  &        &         & y^4z^2, \\
                                  &        & xy^2z^2,& y^3z^3, \\
                                  & x^2z^5
                              \end{array}
             \right\}.
          \]
      Note that $\MO(I) = x^2z^5$, and that $\II_{2} = \{(0,4),(1,2),(0,3),(2,0) \}$. For any element $\alpha \in \II_2$ we have $0 < f_3(\alpha) < \infty$. But $|(1,2)|+f_3(1,2) = 5 < 6 = |(0,4)| + f_3(0,4)$. So the ideal $I$ is not almost reverse lexicographic. Indeed, $y^4z > xy^2z^2$ but $y^4z \notin I$.
   \end{Ex}

   \begin{Coro}\label{Cor:ARLisWLPandSLP}
      Let $I$ be an almost reverse lexicographic ideal in a polynomial ring $R=k[x_1,\ldots,x_n]$. Suppose that the last generator of $I$ is $\MO = x_1^{\omega_1} \cdots x_{\mu}^{\omega_{\mu}}$ for some $(\omega_1,\ldots, \omega_{\mu}) \in \mathbb{Z}^{\mu}_{\ge 0}$ with $\omega_{\mu} > 0$, and that the minimal system $\G$ of generators of $I$ is of the form given in {\rm (\ref{Set:MGS})}. Let $1 \le i \le \mu-1$. For any $\alpha \in \II_{i-1}$ and $\beta \in \II_{i}$, we have
         \[
             |\alpha| + f_i(\alpha) \le f_{i}(0) \le |\beta| + f_{i+1}(\beta),
         \]
      where $0 \in \ZZ^{i-1}$. In particular if $N, M$ are minimal generators of $I$ with $\max N < \max M$, then $\deg N \le \deg M$.
   \end{Coro}
   \begin{proof}
      We claim that $|\alpha| < f_i(0)$. Indeed, if $|\alpha| \ge f_i(0)$, then $x_i^{|\alpha|} \in I$. Since $I$ is strongly stable, $x^{\alpha}= x_1^{\alpha_1} \cdots x_{i-1}^{\alpha_{i-1}} \in I$. But then $\alpha$ can't be an element of $\II_{i-1}$ by Remark \ref{Remk:When_F_be_Zero} and Lemma \ref{Lem:Property_Of_F's_Ext}. This shows that $|\alpha| < f_i(0)$. From the strongly stableness of $I$, then one can see that $x_{i-1}^{|\alpha|} x_{i}^{f_i(0) - |\alpha|}$, and hence $x^{\alpha} x_{i}^{f_i(0) - |\alpha|}$ is also contained in $I$. By the definition of $f_i$, the first inequality follows.

      For the second inequality, consider the $i$-tuple $\delta=(0,\ldots,0,|\beta|+1)$. Note that $\delta > \beta$. If $\delta \notin \II_i$, then $|\beta| + 1 \ge f_{i}(0)$, since the $(i-1)$-tuple $0$ is contained in $\II_{i-1}$. This shows that $x_i^{|\beta|+1} \in I$. It follows that $f_i(0) \le |\beta| + 1 \le |\beta| + f_{i+1}(\beta)$, thus we are done. Now suppose that $\delta \in \II_i$. Since $x^{\delta}x_{i+1}^{f_{i+1}(\delta)} = x_{i}^{|\beta|+1} x_{i+1}^{f_{i+1}(\delta)} \in I$, the monomial $x_i^{|\beta|+1 + f_{i+1}(\delta)}$ is also contained in $I$. Hence we have
      \begin{align*}
          f_i(0) &\le |\beta|+ 1 + f_{i+1}(\delta) = |\beta| + 1 + f_{i+1}(0,\ldots,0,|\beta|+1)\\
                 &\le |\beta| + f_{i+1}(\beta),
      \end{align*}
      where the last inequality follows from Proposition \ref{Prop:EquivCondARL}.
   \end{proof}
   \vskip 2cm

%%%%%%%%%%%%%%%%%%%%%%%%%%%%%%%%%%%%%%%%%%%%%%%%%%%%%%%%%%%%%%%%%%%%%%%%%%%%%%%%%%%%%%%%%%%%%%%%%%%%%%%%%%%%%%
%%%%%%%%%%%%%%                  End of Minimal System of Generators
%%%%%%%%%%%%%%%%%%%%%%%%%%%%%%%%%%%%%%%%%%%%%%%%%%%%%%%%%%%%%%%%%%%%%%%%%%%%%%%%%%%%%%%%%%%%%%%%%%%%%%%%%%%%%%

%%%%%%%%%%%%%%%%%%%%%%%%%%%%%%%%%%%%%%%%%%%%%%%%%%%%%%%%%%%%%%%%%%%%%%%%%%%%%%%%%%%%%%%%%%%%%%%%%%%%%%%%%%%%%%%%%
%%%%%%%%%%%%%%%%%%%%%%%%%%%%%%%%%%%%%%%%%%%%%%%%%%%%%%%%%%%%%%%%%%%%%%%%%%%%%%%%%%%%%%%%%%%%%%%%%%%%%%%%%%%%%%%%%
\section{\sc Hilbert Functions}
%%%%%%%%%%%%%%%%%%%%%%%%%%%%%%%%%%%%%%%%%%%%%%%%%%%%%%%%%%%%%%%%%%%%%%%%%%%%%%%%%%%%%%%%%%%%%%%%%%%%%%%%%%%%%%%%%
%%%%%%%%%%%%%%%%%%%%%%%%%%%%%%%%%%%%%%%%%%%%%%%%%%%%%%%%%%%%%%%%%%%%%%%%%%%%%%%%%%%%%%%%%%%%%%%%%%%%%%%%%%%%%%%%%

   In this section we study Hilbert functions of almost reverse lexicographic ideals in a polynomial ring over a field of arbitrary characteristic. Then we give a criterion for a sequence of nonnegative integers to be induced from an almost reverse lexicographic ideal in a polynomial ring.

   \begin{Defn}
      A sequence $\HB=(h_0, h_1, h_2, \ldots)$ of nonnegative integers is said to be induced from a homogeneous ideal $I$ in a polynomial ring $R=k[x_1,\ldots,x_{n}]$ if $h_d = H(R/I,d)$ for any $d \ge 0$. In this case, we also say that the Hilbert function of $R/I$ is given by the sequence $\HB$.
   \end{Defn}

   Let $\HB = (h_0, h_1, h_2, \ldots)$ be a sequence of nonnegative integers with $h_0 = 1$. Now we will define some sequences induced from the sequence $\HB$. First we denote the sequence $\HB$ by $h^{(0)}_\bullet$, i.e.\ $h^{(0)}_d = h_d$ for all $d$. Then for $1 \le i < h_1$, we define the sequences $h^{(i)}_\bullet$ inductively by
   \[
      h^{(i)}_d = \max\{ 0, \ h^{(i-1)}_d - h^{(i-1)}_{d-1}\} \text{ for } d \ge 1,
   \]
   and $h^{(i)}_0 = 1$. Next we set
   \[
      r_i = \min\{ \ d \ | \ h^{(i)}_{d} \le h^{(i)}_{d-1} \} \text{ for } 0 \le i < \max\{1,h_1\},
   \]
   and we define
   \[
      D(\HB) = \min\{ \ i \ | \ r_{i} < \infty \}.
   \]

   \begin{Remk}\label{Remk:D_and_r}
     Suppose that $\HB = (h_0, h_1, h_2, \ldots)$ is a sequence of nonnegative integers with $h_0 = 1$.
     \begin{enumerate}
     \item If $0 \le h_1 \le 1$, then $r_0 = 1$ since $h_0=1$. Hence $D(\HB) = 0$.
     \item Suppose $h_1 > 1$. Recall that $h^{(i)}_0 = 1$. Since $h^{(i)}_1 = h_1 - i$ for any $1 \le i < h_1$, we have $r_{h_1-1} = 1$. Thus $D(\HB) \le h_1-1$.
     \end{enumerate}
     This shows that $D(\HB) \le h_1$, and that the condition in the following definition is not an empty condition.
   \end{Remk}

   \begin{Defn}\label{Def:UniTail}
      A sequence $\HB$ of nonnegative integers with $h_0 = 1$ is said to be unimodal at each tail if $\HIB$ is unimodal for any $D(\HB) \le i < \max\{1, h_1\}$, i.e.\
         \[
             h^{(i)}_d \le h^{(i)}_{d-1} \text{ for all } d \ge r_i.
         \]
   \end{Defn}

   The following lemma, which is introduced in the papers \cite{ACP,AS}, gives information on the Hilbert function of an almost reverse lexicographic ideal, together with its corollary. Its corollary is similar with the one given in the paper\cite{ACP}, but we don't need to assume that $R/I$ is Artinian(See Corollary 3.3 in the paper\cite{ACP}).

   \begin{Lem}\cite{ACP,AS}\label{Lem:Cond_Belong_In_G(ginI)}
       Let $I$ be a strongly stable ideal in the polynomial ring $R=k[x_1,\ldots,x_n]$ with $\max \MO(I) = n$. If $M$ is a monomial in $(I:x_n)$ of degree $d$ whose canonical image in $(I:x_n)/I$ is not zero, then $M x_n \in \G$. In particular,
       \[
          \dim_k ((I:x_n)/I)_d = |\{ N \in \G \ | \, \max N = n, \deg N = d+1 \}|.
       \]
   \end{Lem}
   \begin{proof}
       Suppose on the contrary that $M x_n$ is not a minimal generator of $I$. Then there are monomials $G \in I$ and $N \in R$ such that $GN = M x_n$ with $\deg N \ge 1$. If $x_n$ divides $N$, then $M = G(N/x_n) \in I$, a contradiction. This shows that $x_n$ must divide $G$ but not $N$. Since $N$ is a monomial of positive degree, there is an integer $j < n$ such that $x_j$ divides $N$. Thus we have $M = (N/x_j)[x_j(G/x_n)] \in I$, since $I$ is strongly stable. This contradicts to $M \notin I$. Hence $Mx_n$ is a minimal generator of $I$.

       Conversely, if $N$ is a minimal generator of $I$ such that $\max N = n$, then $N/x_n$ is a monomial in $(I:x_n)$ whose image in $I$ is not zero. So, the second assertion follows immediately from the first.
   \end{proof}

   \begin{Coro}{\rm(cf.\ \cite{ACP})}\label{Cor:HilbertF_of_ARL_1st}
      Let $I$ be an almost reverse lexicographic ideal in $R=k[x_1,\ldots,x_n]$. Suppose that the last generator of $I$ is $\MO = x^{\omega}$ for some $\omega = (\omega_1,\ldots,\omega_{\mu}) \in \ZZ^{\mu}_{\ge 0}$ with $\omega_{\mu} > 0$, and that the minimal system $\G$ of generators of $I$ is of the form given in {\rm(\ref{Set:MGS})}. If $\mu = n$, then we have:
      \begin{enumerate}
        \item For any $d \ge f_{n-1}(0)$,
             \begin{align*}
                  H(R/I, d-1) &- H(R/I, d) \\
                                  &= |\{ M \in \G \ | \, \max M = n, \deg M = d \}|.
             \end{align*}
             In particular, $H(R/I,d) \le H(R/I,d-1)$ if $d \ge f_{n-1}(0)$.
        \item $H(R/I,d) > H(R/I,d-1)$ for any $d < f_{n-1}(0)$.
        \item $f_{n-1}(0) = \min\{d | H(R/I,d) \le H(R/I,d-1) \}$.
        \item $H(R/I+(x_{n}),d) = \max\{0,H(R/I,d)-H(R/I,d-1)\}$ for any $d$.
      \end{enumerate}
   \end{Coro}
   \begin{proof}
      From the exact sequence
      \[
          0 \RA ((I:x_{n})/I)_{d-1} \RA (R/I)_{d-1} \xrightarrow[\qquad]{\times x_{n}}
         (R/I)_{d} \RA (R/I+(x_{n}))_{d} \RA 0,
      \]
      we have
      \[
           H(R/I,d) - H(R/I,d-1) = H(R/I+(x_{n}),d) - \dim_k ((I:x_{n})/I)_{d-1}.
      \]

      If $d \ge f_{n-1}(0)$, then $x_{n-1}^{d}$ is contained in $I$. Hence for any $(n-1)$-tuple $\alpha=(\alpha_1,\ldots,\alpha_{n-1}) \in \ZZ^{n-1}_{\ge 0}$ with $|\alpha| = d$, $I$ contains the monomial $x^\alpha$ because $I$ is strongly stable. This shows that $(R/I + (x_{n}))_d = 0$. It follows that $H(R/I,d-1) - H(R/I,d) = \dim_k ((I:x_{n})/I)_{d-1}$. From Lemma \ref{Lem:Cond_Belong_In_G(ginI)}, the first assertion follows.

      On the other hand, suppose that $d < f_{n-1}(0)$. Then there is no minimal generator $M$ of $I$ with $\max M = n$ and $\deg M = d$ by Corollary \ref{Cor:ARLisWLPandSLP}. This implies that $((I:x_{n})/I)_{d-1} = 0$ by Lemma \ref{Lem:Cond_Belong_In_G(ginI)}. Hence we have $H(R/I,d) - H(R/I,d-1) = H(R/I+(x_{n}),d)$. Now we will show that $H(R/I+(x_n),d) > 0$. If $H(R/I+(x_{n}),d) = 0$, then $I$ must contain $x_{n-1}^d$. Hence $d \ge f_{n-1}(0)$. This shows that $H(R/I+(x_n), d) > 0$ if $d < f_{n-1}(0)$. Therefore the second assertion holds.

      The third assertion is followed clearly from the first and second assertions. Using the assertions (1), (2) and the exact sequence at the beginning, we can see that the fourth assertion follows.
   \end{proof}

   Now we will show that a sequence $\HB$ of nonnegative integers is unimodal at each tail if and only if it is induced from an almost reverse lexicographic ideal. To show this, we regard the zero ideal as an almost reverse lexicographic ideal.

   \begin{Prop}\label{Prop:ArlImpDecT}
      Let $I$ be an almost reverse lexicographic ideal in a polynomial ring $R=k[x_1,\dots,x_n]$. Then the sequence $\HB=(h_0,h_1,h_2,\ldots)$ induced from $I$ is unimodal at each tail.
   \end{Prop}
   \begin{proof}
      If $I$ is the zero ideal in $R$, then $h^{(i)}_d = \binom{n-(i+1)+d}{d}$ for any $0 \le i \le n-1$. It follows that $D(\HB)=n-1$ and $h^{(n-1)}_{\bullet} = (1,1,1,\ldots)$. Hence $\HB$ is unimodal at each tail.

      Suppose that $I$ is a nonzero ideal in $R$, and that $\MO(I) = x^{\omega}$ for some $\omega = (\omega_1,\ldots,\omega_{\mu}) \in \ZZ^{\mu}_{\ge 0}$ with $\omega_{\mu} > 0$. By Proposition \ref{Prop:MinGenOfARL}, we can assume that the minimal system $\G$ of generators of $I$ is of the form given in {\rm(\ref{Set:MGS})}. If $\deg \MO = 1$, then $I$ is of the form $I= (x_1, \ldots, x_{\mu})$. In this case, $R/I$ is isomorphic to $k[x_{\mu+1},\ldots,x_{n}]$. So the sequence induced from $I$ is unimodal at each tail, as shown in the previous paragraph. Thus we may assume that $\deg \MO \ge 2$. We note that $n-{\mu} < h_1 \le n$ and $1 \le h_1$, because $I$ doesn't contain the monomials $x_{\mu}, x_{\mu+1}, \ldots, x_n$.

      For each $0 \le i < h_1$, we set $R^{(i)}=k[x_1,\ldots,x_{n-i}]$, and $I^{(i)}$ the ideal in $R^{(i)}$ generated by the set
      \[
           \G^{(i)} := \{ M \in \G(I) | \max M \le n-i \}.
      \]
      Then $I^{(i)}$ is an almost reverse lexicographic ideal in $R^{(i)}$ by Proposition \ref{Prop:EquivCondARL}. Furthermore $R^{(0)}/I^{(0)} = R/I$ and $R^{(i)}/I^{(i)}$ is isomorphic to $R/I+(x_n,\ldots,x_{n-i+1})$ for any $1 \le i < h_1$. Thus $R^{(i)}/(I^{(i)}+(x_{n-i}))$ is isomorphic to $R^{(i+1)}/I^{(i+1)}$ for any $0 \le i < h_1-1$. We note that the last generator of $I^{(i)}$ is
      \begin{equation}\label{Eqn:TheLastGenOFReducedIdeal}
          \MO^{(i)} := \MO(I^{(i)}) = \begin{cases}
                                       x^{\omega} = \MO(I) \ &\text{ if } 0 \le i \le n-{\mu}, \\
                                       x_{n-i}^{f_{n-i}(0)} \ &\text{ if } n-{\mu} + 1 \le i < h_1,
                                      \end{cases}
      \end{equation}
      by Corollary \ref{Cor:ARLisWLPandSLP}.

      Under the above setting, we will show that $\HIB$ is induced from the ideal $I^{(i)}$ for any $1\le i < h_1$, and that $\HIB$ is unimodal if $D(\HB) \le i < h_1$. Consider the following exact sequence,
      \begin{align*}
            0 \RA ((I^{(i)}:x_{n-i})&/I^{(i)})_{d-1} \RA \\
                 &(R^{(i)}/I^{(i)})_{d-1} \xrightarrow[\qquad]{\times x_{n-i}} (R^{(i)}/I^{(i)})_{d} \RA (R^{(i+1)}/I^{(i+1)})_{d} \RA 0.
      \end{align*}

      First we claim that if $0 \le i < n - {\mu}$, then $h^{(i)}_d \ge h^{(i)}_{d-1}$ and $H(R^{(i+1)}/I^{(i+1)},d) = h^{(i+1)}_d$ for any $d$. Suppose $0 \le i < n - \mu$. Then $\max \MO^{(i)} = \mu < n - i$ as shown in the equation {\rm(\ref{Eqn:TheLastGenOFReducedIdeal})}. Hence $x_{n-i}$ is $R^{(i)}/I^{(i)}$-regular. It follow that for any $d$
      \[
          H(R^{(i)}/I^{(i)},d) - H(R^{(i)}/I^{(i)},d-1) = H(R^{(i+1)}/I^{(i+1)},d) \ge 0.
      \]
      Since $h^{(0)}_d = H(R^{(0)}/I^{(0)},d)$, we can inductively show that for any $0 \le i < n-{\mu}$,
      \begin{equation}\label{Eqn:h_d-h_d-1}
           h^{(i)}_d = H(R^{(i)}/I^{(i)},d) \ge H(R^{(i)}/I^{(i)},d-1) = h^{(i)}_{d-1},
      \end{equation}
      and hence
      \begin{align*}
          h^{(i+1)}_d & = \max\{0, h^{(i)}_d - h^{(i)}_{d-1}\} = h^{(i)}_d - h^{(i)}_{d-1}  \\
                    & = H(R^{(i)}/I^{(i)},d) - H(R^{(i)}/I^{(i)},d-1)= H(R^{(i+1)}/I^{(i+1)},d).
      \end{align*}
      Until now we have shown that $\HIB$ is induced from the ideal $I^{(i)}$ if $0 \le i \le n-\mu$, and that $h^{(i)}_d \ge h^{(i)}_{d-1}$ for any $d$ if $0 \le i < n-\mu$.

      Suppose that $n-\mu \le i < h_1$. Then $R^{(i)}=k[x_1,\ldots,x_{n-i}]$ and $\max \MO^{(i)} = n-i$. Since the sequence $h^{(n-\mu)}_\bullet$ is induced from the ideal $I^{(n-\mu)}$, one can see by induction that for any $n-\mu \le i < h_1-1$,
      \begin{align*}
          H(R^{(i+1)}/I^{(i+1)}, d)
                          &= \max\{0, H(R^{(i)}/I^{(i)},d)-H(R^{(i)}/I^{(i)},d-1)\} \\
                          &= \max\{0,h^{(i)}_d-h^{(i)}_{d-1}\}  \\
                          &= h^{(i+1)}_d,
      \end{align*}
      where the first equality comes from Corollary \ref{Cor:HilbertF_of_ARL_1st} (4).

      We have to show that $\HIB$ is unimodal if $D(\HB) \le i < h_1$. To find out $D(\HB)$, consider the case that $i < n-\mu$. In this case $\max \MO^{(i)} = \max x^\omega = \mu < n-i$. Hence $x_{n-i}^d$ can't be contained in $I^{(i)}$. It follows that $h^{(i)}_d = H(R^{(i)}/I^{(i)},d) > 0$ for any $d$. Since $h^{(D(\HB)+1)}_{t} = 0$ where $t=r_{D(\HB)}$, $D(\HB)$ should be greater than or equal to $n-\mu-1$.

      Now suppose that $n-\mu \le i < h_1$. Then $R^{(i)}=k[x_1,\ldots,x_{n-i}]$ and $\max \MO^{(i)} = n-i$. By Corollary \ref{Cor:HilbertF_of_ARL_1st} (3), we have $r_i = f_{n-i-1}(0) < \infty$. It follows that $h^{(i)}_\bullet$ is unimodal by Corollary \ref{Cor:HilbertF_of_ARL_1st} (1) and (2).

      Therefore it is enough to show that our assertion holds even if $D(\HB) = n-\mu-1$, i.e.\ $\HIB$ is unimodal when $i=D(\HB)=n-\mu-1$. But in this case, we have already shown that $h^{(i)}_{d-1} \le h^{(i)}_d$ for any $d$ in {\rm(\ref{Eqn:h_d-h_d-1})}. So, we are done.
   \end{proof}

   Let $I \subset R=k[x_1,\ldots,x_n]$ be an almost reverse lexicographic ideal. Suppose that the last generator of $I$ is $\MO(I) = x_{n-1}^t$ for some $t > 0$, i.e.\ $f_{n-1}(0) = t$. Then we have the functions $f^{I}_i$ and the sets $\II^{I}_i$ which are defined for $I$. Now we will make a new almost reverse lexicographic ideal by adding some new generators into $\G(I)$. This method will play an important role in the following theorems. Let $T$ be the set given by
   \[
       T =   \left\{(\alpha_1,\ldots,\alpha_{n-2},\beta) \in \ZZ^{n-1}_{\ge 0}
                  \left| \begin{array}{l}
                         \alpha = (\alpha_1,\ldots,\alpha_{n-2}) \in \II_{n-2}, \\
                         0 \le \beta < f_{n-1}(\alpha)
                       \end{array}
                \right.  \right\}.
   \]
   Suppose that $s$ is a positive integer with $s \le |T|$. Choose $s$ largest elements in $T$, say $A_1, \ldots, A_s$ in $T$ such that $A_1 > A_2 > \cdots > A_s > M$ for any $M \in T - \{ A_1, \ldots, A_s \}$. Let $g:\{A_1,\ldots,A_s\} \RA \ZZ_{>0}$ be a function satisfying
   \begin{equation}\label{Eqn:Property_AuxFunc}
        f_{n-1}(0) = t \le |A_i| + g(A_i) \le |A_j| + g(A_{j}) \text{ for any } A_i > A_j.
   \end{equation}
   Under this circumstance, we have the following lemma.
   \begin{Lem}\label{Lem:Generating_ARL}
      Let $J$ be the ideal in $R$ generated by the set $\SSS := \G(I) \cup \{x^{A_i}x_{n}^{g(A_i)} | 1 \le i \le s \}$.
      Then $J$ is an almost reverse lexicographic ideal with $\G(J) = \SSS$.
   \end{Lem}
   \begin{proof}
     By the choice of the $A_i$'s, every monomial in the set $\SSS$ is a minimal generator of $J$. Indeed, there is no monomial $M$ in $\G(I)$ such that $M$ divides $x^{A_i}x_n^{g(A_i)}$ for any $i$. Hence we have $\G(J) = \SSS$. First we claim that the ideal $J$ is strongly stable. We have to show if $N$ is a monomial in $J$ such that $x_i$ divides $N$, then $x_j(N/x_i) \in J$ for any $j < i$. It suffices to show only when $N$ is a minimal generator of $J$. Since $I$ is almost reverse lexicographic, for the monomials $N$ in $\G(I)$ we have $x_j(N/x_i) \in I \subset J$. Hence it is enough to show that the condition is also satisfied for the new generators $N=x^{A_l}x_{n}^{g(A_l)}$ of $J$. Suppose that $A_l = (\alpha_1,\ldots,\alpha_{n-1})$. There are two possibilities:(1) $i < n$ and $\alpha_i > 0$, in this case let $\delta = (\alpha_1,\ldots,\alpha_j+1,\ldots,\alpha_i-1,\ldots,\alpha_{n-1})$ or (2) $i=n$, in this case we set $\delta = (\alpha_1,\ldots,\alpha_j+1,\ldots,\alpha_{n-1})$. We may assume that $\delta \in T$. Otherwise, $x^{\delta} \in I \subset J$ by the definition of $T$, so nothing is left to show. By the assumption, it follows that $x^{\delta}x_{n}^{g(\delta)} \in J$, since $\delta > A_l$. Now for the case (1), we have
     \[
        g(\delta) \le g(A_l) + |A_l| - |\delta| = g(A_l).
     \]
     Hence $x_j(N/x_i) = x_j(x^{A_l}x_{n}^{g(A_l)}/x_i) = x^{\delta}x_{n}^{g(A_l)} \in J$. For the case (2), we have
     \[
        g(\delta) \le g(A_l) + |A_l| - |\delta| = g(A_l) - 1.
     \]
     Hence $x_j(N/x_n) = x_j(x^{A_l}x_{n}^{g(A_l)}/x_n) = x^{\delta}x_n^{g(A_l) - 1} \in J$. It follows that $J$ is strongly stable. Furthermore, $J$ satisfies the condition in Proposition \ref{Prop:EquivCondARL}, because we have
     \[
        f^{J}_i = \begin{cases}
                       f^{I}_i, &\text{ if } 1 \le i \le n-1, \\
                       g, &\text{ if } i = n,
                  \end{cases}
        \text{ and }
        \II^{J}_i = \begin{cases}
                       \II^{I}_i, &\text{ if } 1 \le i \le n-2, \\
                       \{A_1,\ldots,A_s\}, &\text{ if } i = n-1.
                    \end{cases}.
     \]
     Hence $J$ is an almost reverse lexicographic ideal in $R$ with $\G(J) = \SSS$.
   \end{proof}

   Before we introduce the main theorem of this section, consider the following proposition. It was introduced by Green in the paper \cite{Gr} for the case that $I$ defines a point set in $\PP^2$. Then Cho et al.\ \cite{CCP} generalized his results to the case that $I$ defines an Arithmetically Cohen-Macaulay closed subscheme of any dimension $r$ in $\PP^n$. They assumed (1) that there is no generator $M$ of $I$ with $\max M > n-r$, and (2) that a given ideal is strongly stable. But in our case, we assume that a given ideal is almost reverse lexicographic, because that is enough to use in the rest of this paper.

   \begin{Prop}\cite{CCP,Gr}\label{Prop:Deg}
     Let $I$ be an almost reverse lexicographic ideal in the polynomial ring $R=k[x_1,\ldots,x_n]$. Suppose that the last generator of $I$ is $\MO =x_{n-1}^{t}$ for some positive integer $t > 0$. If $I$ has the minimal system of generators of the form given in {\rm(\ref{Set:MGS})} with $\mu = n-1$, then the Hilbert function of $R/I$ is
     \[
         H(R/I,d) = \sum_{\alpha \in \II_{n-2}} f_{n-1}(\alpha)
     \]
     for any $d \ge t = f_{n-1}(0)$.
   \end{Prop}
   \begin{proof}
    We let
    \[
       T =   \left\{(\alpha_1,\ldots,\alpha_{n-2},s) \in \ZZ^{n-1}_{\ge 0}
                  \left| \begin{array}{l}
                         \alpha = (\alpha_1,\ldots,\alpha_{n-2}) \in \II_{n-2}, \\
                         0 \le s < f_{n-1}(\alpha)
                       \end{array}
                \right.  \right\}.
    \]
    Define the function $g:T \RA \ZZ_{>0}$ by $g(\beta) = t+1-|\beta|$ for $\beta \in T$. Then the function $g$ satisfies the condition in {\rm(\ref{Eqn:Property_AuxFunc})}. By Lemma \ref{Lem:Generating_ARL}, the ideal $J$ in $R$ generated by the set $\SSS = \G(I) \cup \{ x^{\beta} x_n^{g(\beta)} | \beta \in T \}$ is almost reverse lexicographic. We note that if $M$ is a minimal generator of $J$ with $\max M = n$, then $\deg M = t+1$, in particular $\MO(J) = x_n^{t+1}$. Since $J$ is strongly stable, this implies that $H(R/J,t+1) = 0$. By Corollary \ref{Cor:HilbertF_of_ARL_1st}, it follows that
    \begin{align*}
        H(R/I,t) &= H(R/J,t) = H(R/J,t) - H(R/J,t+1) \\
                 &= |\{M \in \G(J) | \max M = n, \deg M = t+1 \}| \\
                 &= |T| = \sum_{\alpha \in \II^{I}_{n-2}} f^{I}_{n-1}(\alpha),
    \end{align*}
    since $f^{J}_{n-1}(0) = f^{I}_{n-1}(0) = t$. On the other hand, because $x_n$ is $R/I$-regular, we have an exact sequence
    \[
        0 \RA (R/I)_{d} \xrightarrow[\qquad]{\times x_{n}} (R/I)_{d+1} \RA (R/I+(x_n))_{d+1} \RA 0.
    \]
    Thus $H(R/I,d+1) = H(R/I,d) + H(R/I+(x_n),d+1)$.  Since $(R/I+(x_n))_{d+1} = 0$ for any $d \ge t$, the assertion follows.
   \end{proof}

   This is the main theorem of this section.

   \begin{Thm}\label{Thm:UAET_Imp_ARL_IDEAL}
      Let $\HB = (h_0, h_1, h_2 , \ldots)$ be a sequence of nonnegative integers. If $\HB$ is unimodal at each tail, then $\HB$ is induced from an almost reverse lexicographic ideal $I$ in the polynomial ring $R=k[x_1,\ldots,x_n]$ where $n = h_1$ if $h_1 \ge 1$, and $n=1$ if $h_1 = 0$.
   \end{Thm}
   \begin{proof}
      We do induction on $h_1$. Suppose $h_1 \le 1$. Then $r_0 = 1$ because $h_0 = 1$. This implies that the sequence $\HB$ is either $\HB=(1,1,1,\ldots)$ or $\HB=(1,1,\ldots,1,0,0,\ldots)$. For the first case choose $I = 0$, and for the second choose $I = (x_1^{f_1})$, respectively, in $R=k[x_1]$, where $f_1 = \min \{ t | h_t = 0 \}$. Then the assertion follows.

      For general case, suppose that $\HB$ is unimodal at each tail and $h_1 > 1$. Since $h^{(1)}_\bullet$ is also unimodal at each tail, there is an almost reverse lexicographic ideal $J$ in $S=k[x_1,\ldots,x_{n-1}]$ such that $n-1 = h^{(1)}_1$ and $h^{(1)}_d = H(S/J,d)$ for all $d$, by the induction hypothesis. Then $n = h_1$. Set $I$ the ideal generated by $\G(J)$ in the ring $R=k[x_1,\ldots,x_n]$. Then $I$ is also almost reverse lexicographic by Proposition \ref{Prop:EquivCondARL}. Furthermore $R/I+(x_n)$ is isomorphic to $S/J$. Since $x_n$ is $R/I$-regular, we have an exact sequence
      \[
          0 \RA (R/I)_{d-1} \xrightarrow[\qquad]{\times x_{n}} (R/I)_{d} \RA (S/J)_{d} \RA 0.
      \]
      It follows that $H(R/I,d) - H(R/I,d-1) = H(S/J,d) = h^{(1)}_d$ for all $d \ge 1$. This shows that $H(R/I,d) \ge H(R/I,d-1)$ for all $d \ge 1$. Hence for any $d < r_0$,
      \begin{align*}
          h_d &= \sum_{i=1}^{d} (h_i - h_{i-1}) + h_0 = \sum_{i=1}^{d} h^{(1)}_i + h_0\\
           &= \sum_{i=1}^{d} (H(R/I,i) - H(R/I,i-1)) + H(R/I,0) = H(R/I,d),
      \end{align*}
      because $H(R/I,0) = 1 = h_0$. If $r_0 = \infty$, then we are done, i.e.\ the sequence $\HB$ is induced from the almost reverse lexicographic ideal $I$ in $R=k[x_1,\ldots,x_n]$ with $n=h_1$.

      Thus we may assume $r_0 < \infty$. Then $D(\HB)=0$ and $H(S/J,r_0) = h^{(1)}_{r_0} = 0$ by the definition of $r_0$. Since $\HB$ is unimodal at each tail, $h_d \le h_{d-1}$ for any $d \ge r_0$. Furthermore, it also follows that $x_{n-1}^{r_0} \in \G(J)$. Indeed, from $H(S/J,r_0) = h^{(1)}_{r_0} = 0$, we have $x_{n-1}^{r_0} \in J$. Hence, $f_{n-1}(0) \le r_0$. Furthermore if $d < r_0$, then $H(S/J,d) = h^{(1)}_d > 0$ by the definition of $r_0$. It follows that $x_{n-1}^d \notin J$ if $d < r_0$. Therefore $f_{n-1}(0) \ge r_0$. This shows that $r_0 = f_{n-1}(0)$, and hence $x_{n-1}^{r_0} \in \G(J) = \G(I)$. By Proposition \ref{Prop:Deg}, then we have $H(R/I,d) = \sum_{\alpha \in \II_{n-2}} f_{n-1}(\alpha)$ for any $d \ge r_0$. Recall that $H(R/I,d) \ge H(R/I,d-1)$ for any $d \ge 1$, and that $H(R/I,d) = h_d$ for any $d \le r_0-1$. Summing up, we have
      \[
         H(R/I,d) = H(R/I,r_0) \ge H(R/I,r_0-1) = h_{r_0 - 1} \ge h_{r_0},
      \]
      for any $d \ge r_0$, where the last inequality follows from the definition of $r_0$.

      Using the method in Lemma \ref{Lem:Generating_ARL}, we will construct an almost reverse lexicographic ideal $K$ in $R$ such that $\HB$ is induced from $K$. First we set $i:=0$, $K_i := I$, $d_{i-1} := r_0$, $\II^{(i)} := \emptyset$ and
      \[
          T_i :=   \left\{(\alpha_1,\ldots,\alpha_{n-2},\alpha_{n-1}) \in \ZZ^{n-1}_{\ge 0}
                     \left| \begin{array}{l}
                            \alpha = (\alpha_1,\ldots,\alpha_{n-2}) \in \II_{n-2}, \\
                            0 \le \alpha_{n-1} < f_{n-1}(\alpha)
                          \end{array}
                   \right.  \right\}.
      \]
      Then by Proposition \ref{Prop:Deg}, $|T_i| = H(R/K_i,d_{i-1}) = H(R/K_i,d)$ for all $d \ge d_{i-1}$.

      Now start the process: Note that $h_d = H(R/K_i,d)$ for all $d < d_{i-1}$. Since $\HB$ is unimodal at each tail, for any $d \ge d_{i-1} \ge r_0$, it follows that
      \[
          H(R/K_i,d) = H(R/K_i,d_{i-1}) \ge h_{d_{i-1}} \ge h_d.
      \]
      We set $d_i = \min\{d | d \ge d_{i-1}, H(R/K_i,d) > h_d \}$. If $d_i = \infty$, then the almost reverse lexicographic ideal $K_i$ is the ideal we want to construct. Suppose that $d_i < \infty$. Let $t_i = H(R/K_i,d_i) - h_{d_i}$. Choose $t_i$ largest elements in $T_i$, say $A_1 > \ldots > A_{t_i}$. Since $|T_i| = H(R/K_i,d_{i-1}) = H(R/K_i,d_i) \ge t_i$, one can always choose such $t_i$ elements from $T_i$. For each $1 \le j \le t_i$, define $g(A_{j}) = d_i - |A_{j}|$. Let $K_{i+1}$ be the ideal generated by
      \[
          \G(K_{i+1}) := \G(K_i) \cup \{ x^{A_{j}} x_{n}^{g(A_{j})} | 1 \le j \le t_i \}.
      \]
      Set $\II^{(i+1)} := \II^{(i)} \cup \{A_{1}, \ldots, A_{t_i} \}$, and set $T_{i+1} := T_{i} - \{A_1,\ldots,A_{t_i}\}$. To show that $K_{i+1}$ is almost reverse lexicographic, we must prove that $g$ satisfies the condition in {\rm(\ref{Eqn:Property_AuxFunc})}, i.e.\ if $A,B$ are elements in $\II^{(i+1)}$ with $A > B$, then $r_0 \le |A| + g(A) \le |B| + g(B)$. But, by the definition of the $\II^{(i)}$'s, it is enough to show that if $B \in \II^{(i)}-\II^{(i-1)}$, then $r_0 \le |B| + g(B) \le |A_j| + g(A_j) \le |A_{j+1}| + g(A_{j+1})$ for any $1 \le j \le t_i-1$. Indeed, since we have
      \begin{align*}
        f_{n-1}(0) = r_0 &\le d_{i-1} = |B| + g(B) \\
            &\le d_{i} = |A_j| + g(A_j) = |A_{j+1}| + g(A_{j+1}),
      \end{align*}
      it follows from Lemma \ref{Lem:Generating_ARL} that the ideal $K_{i+1}$ is almost reverse lexicographic. By the choice of $K_{i+1}$ and by Corollary \ref{Cor:HilbertF_of_ARL_1st} (1), we have
      \[
          H(R/K_{i+1},d) = \begin{cases}
                               H(R/K_{i},d) = h_d & \text{ if } d < d_i, \\
                               h_{d_i} & \text{ if } d \ge d_i.
                           \end{cases}
      \]
      Furthermore $|T_{i+1}| = |T_i| - t_i = |T_i| - (H(R/K_i,d_i) - h_{d_i}) = h_{d_i} = H(R/K_{i+1},d_{i})$. Increase $i$ by $1$, then repeat this process.

      We have to show that this process stops in a finite number of steps. Note that $|T_i| < |T_{i-1}|$ for any $1 \le i$. Since $|T_0| < \infty$, the process must be terminated in a finite number of steps.
   \end{proof}

   In next section we will use this theorem to prove each Fr\"{o}berg sequence is induced from an almost reverse lexicographic ideal. Putting Proposition \ref{Prop:ArlImpDecT} and Theorem \ref{Thm:UAET_Imp_ARL_IDEAL} together, we have the following criterion.

   \begin{Coro}\label{Coro:Criterion_For_ARL_UAET}
      For a sequence $\HB$ of nonnegative integers, $\HB$ is unimodal at each tail if and only if there is an almost reverse lexicographic ideal in a polynomial ring whose Hilbert function is given by the sequence $\HB$. $\qed$
   \end{Coro}

   \begin{Ex}\label{Ex:First}
   \begin{enumerate}
      \item Suppose that $\HB =(1,1,1,0,\ldots)$. So $\HB$ is unimodal at each tail. Note that the ideal generated by $\G= \{x^3\}$ induces the sequence $\HB$.
      \item Suppose that $\HB = (1,2,3,2,1,0,\ldots)$. Then $h^{(1)}_\bullet=(1,1,1,0,\ldots)$. So $\HB$ is unimodal at each tail. The ideal $J=(x^3) \subset k[x]$ induces the sequence $h^{(1)}_\bullet$. Let $T_0 = \{2,1,0\}$, and $I \subset R=k[x,y]$ the ideal generated by $\G(J)$. Note that $2 > 1 > 0$ in the degree reverse lexicographic order, and $r_0 = 3$. Since $H(R/I,\bullet)=(1,2,3,3,\ldots)$, we have $d_0 = 3$ and $t_0 = H(R/I,d_0) - h_{d_0} = 1$. Define $g(2) = d_0 - |2| = 1$. Set $\G(K_1) = \G(I) \cup \{x^2y^{g(2)}=x^2y\}$, $\II^{(1)} = \{2\}$ and $T_{1} = \{1,0\}$. Note that $H(R/K_{1},\bullet) = (1,2,3,2,2,\ldots)$. Repeat the process, then we have
          \begin{align*}
             d_1 &= 4, t_1 = 1, g(1) = 4-|1| = 3, \G(K_2) = \G(K_1) \cup \{x y^3\}, \\
             \II^{(2)} &=\{2,1\}, T_{2} = \{0\}, \text{ and } H(R/K_{2},\bullet) = (1,2,3,2,1,1,\ldots).
          \end{align*}
          Repeat the process one more time, then
          \begin{align*}
             d_2 &= 5, t_2 = 1, g(1) = 5-|0| = 5, \G(K_3) = \G(K_2) \cup \{y^5\}, \\
             \II^{(3)} &=\{2,1,0\}, T_{3} = \emptyset, \text{ and } H(R/K_{3},\bullet) = (1,2,3,2,1,0,\ldots) = \HB.
          \end{align*}
          Set $K=K_3$. Thus one can see that the ideal generated by
          \[
             \G(K) = \{x^3, x^2y, xy^3, y^5\}
          \]
          induces the sequence $\HB$.
      \item Suppose that $\HB = (1,3,6,8,9,9,6,5,5,\ldots)$. Then it follows that $h^{(1)}_\bullet=(1,2,3,2,1,0,\ldots)$ and $h^{(2)}_\bullet=(1,1,1,0,\ldots)$. So $\HB$ is unimodal at each tail. As shown in (2), the ideal $J \subset S=k[x,y]$ generated by $\{x^3, x^2y, xy^3, y^5\}$ induces the sequence $h^{(1)}_\bullet$.

          Let $T_0 = \{(0,4),(1,2),(0,3),(2,0),(1,1),(0,2),(1,0),(0,1),(0,0)\}$, and $I \subset R=k[x,y,z]$ the ideal generated by $\G(J)$. Note that $(0,4) > (1,2) > \cdots > (0,1) > (0,0)$ in the degree reverse lexicographic order, and $r_0 = 5$. Since $H(R/I,\bullet)=(1,3,6,8,9,9,9,\ldots)$, we have $d_0 = 6$ and $t_0 = H(R/I,d_0) - h_{d_0} = 3$. Define $g(0,4) = d_0 - |(0,4)| = 2, g(1,2) = 3$ and $g(0,3) = 3$.

          Set $\G(K_1) = \G(I) \cup \{y^4z^2, xy^2z^3, y^3z^3\}$, $\II^{(1)} = \{(0,4),(1,2),(0,3)\}$ and $T_{1} = \{(2,0),(1,1),(0,2),(1,0),(0,1),(0,0)\}$. Note that $H(R/K_{1},\bullet) = (1,3,6,8,9,9,6,6,\ldots)$. Repeat this process, then we have
          \begin{align*}
             d_1 = 7, t_1 =& 1, g(2,0) = 7-|2| = 5, \G(K_2) = \G(K_1) \cup \{x^2 z^5\}, \text{ and } \\
             H(R/K_{2},\bullet) &= (1,3,6,8,9,9,6,5,5,\ldots) = \HB.
          \end{align*}
          Set $K=K_2$. Then the ideal generated by
          \[
             \G(K) = \left\{
                          \begin{array}{cccc}
                              x^3,& x^2y,  & xy^3,   & y^5, \\
                                  &        &         & y^4z^2, \\
                                  &        & xy^2z^3,& y^3z^3, \\
                                  & x^2z^5
                              \end{array}
             \right\}
          \]
          induces the sequence $\HB$.
   \end{enumerate}
   \end{Ex}

   \vskip 2cm

%%%%%%%%%%%%%%%%%%%%%%%%%%%%%%%%%%%%%%%%%%%%%%%%%%%%%%%%%%%%%%%%%%%%%%%%%%%%%%%%%%%%%%%%%%%%%%%%%%%%%%%%%%%%%%
%%%%%%%%%%%%%%                  End of Hilbert Functions
%%%%%%%%%%%%%%%%%%%%%%%%%%%%%%%%%%%%%%%%%%%%%%%%%%%%%%%%%%%%%%%%%%%%%%%%%%%%%%%%%%%%%%%%%%%%%%%%%%%%%%%%%%%%%%

%%%%%%%%%%%%%%%%%%%%%%%%%%%%%%%%%%%%%%%%%%%%%%%%%%%%%%%%%%%%%%%%%%%%%%%%%%%%%%%%%%%%%%%%%%%%%%%%%%%%%%%%%%%%%%%%%
%%%%%%%%%%%%%%%%%%%%%%%%%%%%%%%%%%%%%%%%%%%%%%%%%%%%%%%%%%%%%%%%%%%%%%%%%%%%%%%%%%%%%%%%%%%%%%%%%%%%%%%%%%%%%%%%%
\section{\sc Fr\"{o}berg sequences and almost reverse lexicographic ideals}
%%%%%%%%%%%%%%%%%%%%%%%%%%%%%%%%%%%%%%%%%%%%%%%%%%%%%%%%%%%%%%%%%%%%%%%%%%%%%%%%%%%%%%%%%%%%%%%%%%%%%%%%%%%%%%%%%
%%%%%%%%%%%%%%%%%%%%%%%%%%%%%%%%%%%%%%%%%%%%%%%%%%%%%%%%%%%%%%%%%%%%%%%%%%%%%%%%%%%%%%%%%%%%%%%%%%%%%%%%%%%%%%%%%

In this section, we show that every Fr\"{o}berg sequence is unimodal at each tail. Let $P = \sum_{i=0}^{\infty} p_i z^i \in \ZZ[[z]]$ be a formal power series, and let $t = \min\{ d | p_d \le 0 \}$. Recall that by $|P| = | \sum_{i=0}^{\infty} p_i z^i|$, we mean the series $\sum_{i=0}^{\infty} q_i z^i$, where each $q_i \in  \ZZ$ is defined to be
\[
    q_i = \begin{cases}
              p_i, & \text{ if } i < t, \\
              0, & \text{ if } i \ge t.
          \end{cases}
\]

\begin{Defn}\label{Def:FRO_Seq}
   A sequence $\HB=(h_0,h_1,h_2,\ldots)$ is said to be a Fr\"{o}berg sequence if there are nonnegative integers $n,m$ and positive integers $d_1,\ldots,d_m$ such that
   \[
       \sum_{i=0}^{\infty} h_i z^i = \left| \frac{(1-z^{d_1})\cdots(1-z^{d_m})}{(1-z)^n} \right|.
   \]
   In this case we denote it by $\HB=|n;d_1,\ldots,d_m|$ if $m \ge 1$, and by $\HB=|n;\emptyset|$ if $m=0$.
\end{Defn}

To avoid making the same hypothesis in each theorems, if a sequence $\HB$ is given by a Fr\"{o}berg sequence $|n;d_1,\ldots,d_m|$, then we always assume that $n, m$ are nonnegative integers, and assume that $d_1,\ldots,d_m$ are positive integers.

\begin{Ex} \label{Ex:Trivial_Case_of_Fro_seq}
   \begin{enumerate}
      \item If $\HB=|0;d_1,\ldots,d_m|$, then $\HB=(1,0,\ldots)$. Thus $\HB$ is unimodal at each tail.
      \item If $\HB=|n;\emptyset|$, then $h_i = \binom{n-1+i}{i}$, where $\binom{a}{b} = 0$ for any integers $a, b$ with $a < b$. Hence $\HB$ is unimodal at each tail as shown in the proof of Proposition \ref{Prop:ArlImpDecT}.
   \end{enumerate}
\end{Ex}

For a sequence $\HB=(h_0,h_1,\ldots)$, recall that $r_i(\HB)$ is defined to be
\[
   r_i(\HB) = \min \{ d | h^{(i)}_d \le h^{(i)}_{d-1} \},
\]
for $0 \le i < \max \{1, h_1\}$ in the previous section. For a given Fr\"{o}berg sequence $\GB$, the following lemma shows when the induced sequence $g^{(1)}_\bullet$ can be a Fr\"{o}berg sequence.

\begin{Lem}\label{Lem:BasicP_Frob}
   Suppose that $\GB=|n;d_1,\ldots,d_m|$ is a Fr\"{o}berg sequence.
   \begin{enumerate}
      \item If $n \ge 1$, then $|n;d_1,\ldots,d_m,1| = |n-1;d_1,\ldots,d_m|$.
      \item For a positive integer $s$, if $\HB = |n;d_1,\ldots,d_m,s|$, then
            \[
                h_d = \begin{cases}
                         \max \{0, g_d - g_{d-s} \} = g_d - g_{d-s} > 0,  & \text{ if }  d < t, \\
                         0, & \text{ if } d \ge t,
                      \end{cases}
             \]
             where $t = \min \{ d \ge 0 \, | \, g_d \le g_{d-s} \}$ and $g_j = 0$ for $j<0$.
      \item Suppose $n \ge 1$. Then $g^{(1)}_{\bullet} = |n-1;d_1,\ldots,d_m|$ if and only if $g_d \le g_{d-1}$ for any $d \ge r_0(\GB)$.
   \end{enumerate}
\end{Lem}
\begin{proof}
   \begin{enumerate}
      \item It is clear by the definition of a Fr\"{o}berg sequence.
      \item Suppose that
            \[
                 F(z)=\sum_{i=0}^{\infty} \alpha_i z^i = \frac{(1-z^{d_1})\cdots(1-z^{d_{m}})}{(1-z)^n}.
            \]
            Let $p=\min\{ d \, | \, \alpha_d \le 0\}$. Then we have
            $g_i = \begin{cases}
                              \alpha_i, & \text{ if } i < p, \\
                              0, & \text{ otherwise.}
                        \end{cases}$ Furthermore
            \begin{align*}
                \sum_{i=0}^{\infty} h_i z^i &=\left| \frac{(1-z^{d_1})\cdots(1-z^{d_{m}})(1-z^s)}{(1-z)^n} \right| \\
                    &= \left| F(z)(1-z^s) \right| = \left| F(z) - F(z)z^s \right| \\
                    &= \left| \sum_{i=0}^{\infty} \alpha_i z^i - \sum_{i=0}^{\infty} \alpha_i z^{i+s}
                       \right|
                     = \left| \sum_{i=0}^{\infty} (\alpha_{i} - \alpha_{i-s}) z^i\right|,
            \end{align*}
            where we set $\alpha_{j} = 0$ for $j<0$.
            We let $\tilde{t} = \min \{ d \ge 0\, |\, \alpha_d \le \alpha_{d-s} \}$, then
            \[
                 h_d = \begin{cases}
                              \alpha_d - \alpha_{d-s}, & \text{ if }  d < \tilde{t}, \\
                              0, & \text{ otherwise. }
                       \end{cases}
            \]

            Now if $p = \infty$, i.e.\ $\alpha_d > 0$ for all $d$, then nothing is left to prove. So assume that $p < \infty$. Since $g_p = 0 \le g_{p-s}$, it follows that $t \le p$ by the definition of $t$. Now we will show that $\tilde{t} \le t$. Indeed, if $t=p$, then $\alpha_{t} \le 0 = g_{t} \le g_{t-s} = \alpha_{t-s}$. Or if $t<p$, then $\alpha_{t} = g_{t} \le g_{t-s} = \alpha_{t-s}$. In any case, $\alpha_{t} \le \alpha_{t-s}$. Hence it follows from the definition of $\tilde{t}$ that $\tilde{t} \le t$.

            Then we wish to show that $\tilde{t} = t$. It suffices to show $g_{\tilde{t}} \le g_{\tilde{t}-s}$. Note that $\tilde{t} \le t \le p$. If $\tilde{t} < p$, then $g_{\tilde{t}} = \alpha_{\tilde{t}} \le \alpha_{\tilde{t}-s} = g_{\tilde{t}-s}$. For the case $\tilde{t} = p$, we have $g_{\tilde{t}} = 0 \le g_{\tilde{t}-s}$, so we are done. Since $\tilde{t} = t \le p$, the second assertion follows.
      \item Let $\HB = |n;d_1,\ldots,d_m,1|$. By (1), it suffices to show $\HB = g^{(1)}_{\bullet}$ if and only if $g_d \le g_{d-1}$ for any $d \ge r_0(\GB)$. Since this is the case $s=1$ and $t=r_0(\GB)$ in (2), $\HB =  g^{(1)}_{\bullet}$ if and only if $0 = \max\{0, g_d-g_{d-1}\}$ for all $d \ge r_0(\GB)$, or equivalently, $g_d \le g_{d-1}$ for all $d \ge r_0(\GB)$, this is what we want to show.
      \end{enumerate}
\end{proof}

\begin{Remk}\label{Remk:h_1_of_Fro}
   If $\HB=|n;d_1,\ldots,d_m|$ is a Fr\"{o}berg sequence, then $h_1 \le n$. Indeed, we can show this by induction on $m$. If $m=0$, then $h_1 = \binom{n-1+1}{1} = n$. If $m > 1$, then by Lemma \ref{Lem:BasicP_Frob} (2), we have $h_1 = \max\{0, g_1 - g_{1-d_m}\}$, where $\GB=|n;d_1,\ldots,d_{m-1}|$. Hence $h_1 \le g_1 \le n$ by the induction hypothesis.
\end{Remk}

\begin{Lem}\label{Lem:Red_UaT_ALSO_Uat}
   Let $\GB = |n;d_1,\ldots,d_m|$ and $\HB = |n;d_1,\ldots,d_m,s|$ be Fr\"{o}berg sequences. Suppose that $\GB$ is unimodal at each tail.  If there is a positive integer $i$ with $h_i \le h_{i-1}$, then $h_d \le h_{d-1}$ for any $d \ge i$. In this case, $h^{(1)}_{\bullet} = |n-1;d_1,\ldots,d_m,s|$ if $n \ge 1$.
\end{Lem}
\begin{proof}
   If we set $\nu = \min\{ j \ge 0 \,| \, g_j \le g_{j-s} \}$, then it follows by Lemma \ref{Lem:BasicP_Frob}(2) that
   \begin{equation}\label{Eqn:FroSeqbyOne}
       h_j = \begin{cases}
                 \max \{0, g_j - g_{j-s} \} = g_j - g_{j-s} > 0,  & \text{ if }  j < \nu, \\
                 0, & \text{ if } j \ge \nu.
             \end{cases}
   \end{equation}
   Suppose on the contrary that there is an integer $d > i $ with  $h_d > h_{d-1}$. From the equation {\rm (\ref{Eqn:FroSeqbyOne})}, this implies that $d < \nu$. Hence $h_j = g_j - g_{j-s} > 0$ for any $j \le d$, in particular, for $j = i-1, i, d-1$, and $d$.

   Since $h_d > h_{d-1}$, it follows that $g_d - g_{d-s} > g_{d-1} - g_{d-1-s}$, or equivalently, $g_d - g_{d-1} > g_{d-s} - g_{d-s-1}$. Now if $d \ge r_0(\GB)$, then $0 \ge g_d - g_{d-1} > g_{d-s} - g_{d-s-1}$, hence $d-s \ge r_0(\GB)$. Since $\GB$ is unimodal at each tail, then we have
   \[
        g_d - g_{d-s} = \sum_{j=1}^{s} g_{d-s+j} - g_{d-s+j-1} \le 0,
   \]
   which contradicts to $g_d - g_{d-s} > 0$. This shows that $d < r_0(\GB)$, and hence it also follows that $g^{(1)}_d > g^{(1)}_{d-s}$, since $g_d - g_{d-1} > g_{d-s} - g_{d-s-1}$.

   On the other hand, since $h_{i} \le h_{i-1}$, we have $g_i - g_{i-s} \le g_{i-1} - g_{i-1-s}$, or equivalently, $g_i - g_{i-1} \le g_{i-s} - g_{i-s-1}$. Since $i < d < r_0(\GB)$, this implies that $g^{(1)}_{i} \le g^{(1)}_{i-s}$. Hence $D(\GB) \le 1$ and $r_1(\GB) \le i$. Since $\GB$ is unimodal at each tail, it follows that $g^{(1)}_{d} \le g^{(1)}_{i}$. So we have
   \[
      g^{(1)}_{d-s} < g^{(1)}_{d} \le g^{(1)}_{i} \le g^{(1)}_{i-s}.
   \]
   In particular $g^{(1)}_{d-s} < g^{(1)}_{i-s}$ . Since $i<d$, we should have $r_1(\GB) \le d-s$. Since $\GB$ is unimodal at each tail, this induces $g^{(1)}_{d-s} \ge g^{(1)}_{d}$. But it contradicts to $g^{(1)}_{d-s} < g^{(1)}_{d}$. Hence the first assertion follows.

   The last assertion follows from Lemma \ref{Lem:BasicP_Frob}(3).
\end{proof}

\begin{Prop}\label{Prop:FroSeqIsARL}
   Let $\HB=|n;d_1,\ldots,d_m|$ be a Fr\"{o}berg sequence. Then $\HB$ is induced from an almost reverse lexicographic ideal $K$ in the polynomial ring $R=k[x_1,\ldots,x_l]$, where $l = n$ if $n \ge 1$, and $l=1$ if $n=0$.
\end{Prop}
\begin{proof}
   If $m=0$, then we are done as shown in Example \ref{Ex:Trivial_Case_of_Fro_seq} (2). Hence assume that $m \ge 1$. Note that $n \ge h_1$ by Remark \ref{Remk:h_1_of_Fro}. If $h_1=0$, then $\HB = (1,0,\ldots)$. Hence the sequence $\HB$ is induced from the ideal $K=(x_1,\ldots,x_l)$ in the ring $R=k[x_1,\ldots,x_l]$, where $l=1$ if $n=0$, and $l=n$ if $n \ge 1$. So we may assume that $h_1 \ge 1$. We claim that $\HB$ is unimodal at each tail. If our claim is true, then by Theorem \ref{Thm:UAET_Imp_ARL_IDEAL}, there is an almost reverse lexicographic ideal $J$ in the polynomial ring $S=k[x_{n-h_1+1},\ldots,x_{n}]$ such that $H(S/J,d) = h_d$ for all $d$. If $n=h_1$, then the ideal $J$ is the almost reverse lexicographic ideal we want to find. Suppose that $n > h_1$. Let $K$ be the ideal in $R=k[x_1,\ldots,x_n]$ generated by $\{x_1,\ldots,x_{n-h_1}\} \cup \G(J)$. By Proposition \ref{Prop:EquivCondARL}, then $K$ is an  almost reverse lexicographic ideal such that the Hilbert function of $R/K$ is the same with that of $S/J$. This shows that our assertion holds.

   Now we will show that our claim is true, i.e.\ the sequence $\HB$ with $h_1 \ge 1$ is unimodal at each tail by induction on $n$. Since $n \ge h_1 \ge 1$, we have to prove when $n=1$. Let $t = \min \{d_1,\ldots,d_m\}$. Then we have $\HB = (1,\ldots,1,0,0,\ldots)$, where $h_d = 1$ for any $d < t$, and $h_d = 0$ for any $d \ge t$. So we are done.

   For the general case, suppose $n>1$. We have to show that $\HB = |n;d_1,\ldots,d_m|$ is unimodal at each tail. We do induction on $m$. If $m=1$, then $\HB = |n; d_1|$. Let $I$ be the ideal $(x_1^{d_1})$ in the polynomial ring $R=k[x_1,\ldots,x_n]$. Since $H(R/I,d) = h_d$, we are done. Hence it is enough to show if $\GB=|n;d_1,\ldots,d_m|$ is a Fr\"{o}berg sequence which is unimodal at each tail, then for any positive integer $s$, the sequence $\HB=|n;d_1,\ldots,d_m,s|$ is also unimodal at each tail. By Lemma \ref{Lem:Red_UaT_ALSO_Uat}, we have $h^{(1)}_{\bullet} = |n-1;d_1,\ldots,d_m,s|$. By the induction hypothesis, $h^{(1)}_{\bullet}$ is unimodal at each tail. Hence, in order to show that $\HB$ is unimodal at each tail, it suffices to prove for only the case $D(\HB) = 0$, i.e.\ if $r_0(\HB) < \infty$, then $h_d \le h_{d-1}$ for any $d \ge r_0(\HB)$. But we have already shown it in Lemma \ref{Lem:Red_UaT_ALSO_Uat}. So we are done.
\end{proof}

\begin{Ex}
   \begin{enumerate}
      \item  Suppose that $\HB = \left| 3; 3, 3, 5 \right|$. Then it follows that $\HB = (1,3,6,8,9,8,6,3,1,0,\ldots)$. Using the process in Theorem \ref{Thm:UAET_Imp_ARL_IDEAL}, we can construct the almost reverse lexicographic ideal $K$ in $R=k[x,y,z]$ such that the Hilbert function of $R/K$ is given by the sequence $\HB$. 
          \[
             \G(K) = \left\{
                          \begin{array}{cccc}
                              x^3,& x^2y,  & xy^3,   & y^5, \\
                                  &        &         & y^4z, \\
                                  &        & xy^2z^3,& y^3z^3, \\
                                  & x^2z^5,& xyz^5,&   y^2z^5,\\
                                  &        & xz^7,&    yz^7,\\
                                  &        &      &    z^9
                              \end{array}
             \right\}.
          \]
   \end{enumerate}
\end{Ex}

   \vskip 2cm

%%%%%%%%%%%%%%%%%%%%%%%%%%%%%%%%%%%%%%%%%%%%%%%%%%%%%%%%%%%%%%%%%%%%%%%%%%%%%%%%%%%%%%%%%%%%%%%%%%%%%%%%%%%%%%
%%%%%%%%%%%%%%                  End of Fr\"{o}berg Conjecture and Moreno-Socias Conjecture
%%%%%%%%%%%%%%%%%%%%%%%%%%%%%%%%%%%%%%%%%%%%%%%%%%%%%%%%%%%%%%%%%%%%%%%%%%%%%%%%%%%%%%%%%%%%%%%%%%%%%%%%%%%%%%

%%%%%%%%%%%%%%%%%%%%%%%%%%%%%%%%%%%%%%%%%%%%%%%%%%%%%%%%%%%%%%%%%%%%%%%%%%%%%%%%%%%%%%%%%%%%%%%%%%%%%%%%%%
%%%%%%%%%%%%%%%%%%%%%%%%%%%%%%%%%%%%%%%%%%%%%%%%%%%%%%%%%%%%%%%%%%%%%%%%%%%%%%%%%%%%%%%%%%%%%%%%%%%%%%%%%%
\bibliographystyle{amsalpha}
%%%%%%%%%%%%%%%%%%%%%%%%%%%%%%%%%%%%%%%%%%%%%%%%%%%%%%%%%%%%%%%%%%%%%%%%%%%%%%%%%%%%%%%%%%%%%%%%%%%%%%%%%%
%%%%%%%%%%%%%%%%%%%%%%%%%%%%%%%%%%%%%%%%%%%%%%%%%%%%%%%%%%%%%%%%%%%%%%%%%%%%%%%%%%%%%%%%%%%%%%%%%%%%%%%%%%

\end{document}